\documentclass[10pt]{article}
\usepackage{amssymb}
\usepackage{amsmath}
\usepackage{amsthm}
\usepackage{amsfonts}
\usepackage[ansinew]{inputenc}
\usepackage{amsthm,amsfonts}
\makeatletter \@addtoreset{equation}{section} \makeatother

\newtheorem{theorem}{Theorem}[section]
\newtheorem{corollary}{Corollary}[section]
\newtheorem{lemma}{Lemma}[section]
\newtheorem{proposition}{Proposition}[section]
\newtheorem{remark}{Remark}[section]
\newtheorem{assumption}{Assumption}[section]
\newtheorem{definition}{Definition}[section]

\newtheorem{algorithm}{Algorithm}
\newtheorem{example}{Example}[section]
\makeatletter
\renewcommand*{\@biblabel}[1]{\hfill#1.}
\makeatother

\title{Behavioral Traps and the Equilibrium Problem on Hadamard Manifolds.}

\author{Bento, G. C.  \thanks{IME, Universidade Federal de Goi\'as, Goi\^ania, GO 74001-970, BR ({\tt glaydston@mat.ufg.br}).}
\and 
Cruz Neto, J. X. 
\thanks{CCN, DM, Universidade Federal do Piau\'i,
Terezina, PI 64049-550, BR ({\tt jxavier@ufpi.edu.br}).}
\and
Soares Jr, P.A. \thanks{CCN, DM, Universidade Etadual do Piau\'i, Terezina, PI 64002-150, BR ({\tt pedrosoares@uespi.br}) - {\bf Corresponding author.}
}
\and
Soubeyran, A.  \thanks{Aix-Marseille University (Aix-Marseille School of Economics), CNRS \&  EHESS, FR, ({\tt antoine.soubeyran@gmail.com})}
}
\begin{document}
\maketitle
\begin{abstract}
In this paper we present a sufficient condition for the existence of a solution for an \mbox{equilibrium} problem on an Hadamard manifold and under suitable assumptions on the sectional curvature, we \mbox{propose} a framework for the convergence analysis of a proximal point algorithm to solve this equilibrium \mbox{problem} in finite time. Finally we offer an application to  personal equilibrum problems as behavioral traps \mbox{problems}, using a recent \textquotedblleft variational rationality\textquotedblright \; approach of human behavior.



\end{abstract}
\section{Introduction}
The equilibrium problem EP has been widely studied and is a very active field of research. One of the motivations is that various problems may be formulated as an equilibrium problem, for instance, \mbox{optimization} problems, Nash equilibria problems, complementarity problems, fixed point problems and variational inequality problems. An extensive development can be found in Blum and Oettli~\cite{BO1994}, Bianchi and Schaible~\cite{BS1996} and their references.

An important issue is under what conditions there exists a solution to EP. In the linear setting, several authors have provided results answering this question; see, for instance, Ky Fan~\cite{F1961}, Br\'{e}zis et al.~\cite{BLS1972}, Iusem and Sosa~\cite{IS2003} and Iusem et al.~\cite{IKS2009}. As far as we know,  Colao et al.~\cite{CLMM2012} were the first to provide an existence result for equilibrium problems in a Riemannian context, more accurately,  on Hadamard manifolds, in the case where EP is associated to a monotone bifunction which satisfies a certain coercivity condition. On this issue, following the ideas presented in~\cite{IKS2009},  we have presented a weaker sufficient condition than the coercivity assumption used in~\cite{CLMM2012}, to obtain the existence of solutions of EP. 

Recently, in~\cite{CLMM2012} the authors presented an iterative process, Picard iteration, to approximate a solution of the equilibrium problem on an Hadamard manifold which retrieves the proximal iteration of the proximal point method  studied, for example,  by Moudafi~\cite{M1999}, Konnov~\cite{K2003} and Iusem and Sosa~\cite{IS2010}, both in the linear setting. In this paper, we present a proximal algorithm for EP on an Hadamard manifold with null sectional curvature whose iterative process coincides with the proposed in~\cite[Theorem 4.10]{CLMM2012}. We point out that our contribution here is an extension of the convergence result presented in \cite{CLMM2012} to the case where the bifunction of EP is not necessarily monotone. Note that, under assumption of null sectional curvature, our algorithm retrieves the proximal point method for the variational inequalities problem presented by Tang et al.~\cite{Tang2013} and, hence, the proximal point method for minimization problems (see Ferreira and Oliveira~\cite{FO2002}) or, more generally, the proximal point method for vector fields, both on Hadamard manifolds, see Li et al.~\cite{LLMM2009}. In recent years,  extensions to Riemannian manifolds of concepts and techniques which fit in Euclidean spaces are natural. Some algorithms for solving variational inequalities and minimization problems which involve monotone operators have been extended from Hilbert space framework to more general setting of Riemannian manifolds; see Ferreira et al.~\cite{FLN2005}, \cite{FO2002}, Li et al.~\cite{LLM2010} and Wang et al.~\cite{WLML2010}. One reason for the success of techniques extension to the Riemannian setting is the possibility to transform nonconvex problems in linear context in convex problems by introducing a suitable metric; see  Cruz Neto et al.~  \cite{CFL2002} and Rapcs\'{a}k~\cite{R1997}. In regard to the minimization problem where the constrained set is the whole manifold, Bento and Cruz Neto in~\cite{BC2012} showed that the proximal point method has finite termination in the particular case where the objective function is convex and each minimizer is a weak sharp minimum (see Li et al.~\cite{LMWY2011} for a definition). Following the ideas presented by Moudafi in~\cite{M2007}, we present a condition about the bifunction $F$ that, in particular, retrieves the notion of weak sharp minima, and we prove the finite termination of any sequence generated from our iterative process. In particular, the finite termination result in~\cite{BC2012} is extended to minimization problems whose constrained set is not necessarily the  whole  manifold and we obtain a resulted of finite convergence to the proximal point methods for finding singularities of single valued monotone vector fields (see, for instance~\cite{LLMM2009} and Cruz Neto et al. \cite{CFPN2006}) and, hence, for the  variational inequality problem (see N\'{e}meth~\cite{N2003}) among others.

The organization of our paper is as follows. In Section 2, we give some elementary facts on \mbox{Riemannian} manifolds and convexity needed for reading of this paper. In Section 3, we present a sufficient \mbox{condition} for existence of a solution for the equilibrium problem on Hadamard manifolds under similar \mbox{conditions} required in the linear case. In Section 4, the proximal point algorithm for equilibrium problems on Hadamard manifolds is presented, convergence analysis is derived and, under mild assumption, a finite termination result is proved. In Section 5, we give a behavioral application  to the so called personal equilibrium problems where agents have reference dependent utility functions and variable reference points (see, K\H{o}szegi and Rabin~\cite{KR2006}), in the context of a recent and unifying approach of a lot of stability and change dynamics in
Behavioral Sciences, the \textquotedblleft Variational rationality" approach
of worthwhile stays and changes human behavior (see, Soubeyran~\cite{S2009,S2010}). This approach
focus the attention on three main concepts:
\begin{itemize}
\item [i)] worthwhile single changes, where, for an isolated agent or several
interrelated agents, their motivation to change from the current position to
a new position is higher than some adaptive and satisficing worthwhile to
change ratio, time their resistance to change. Motivation to change refers
to the utility of advantages to change, while resistance to change refers to
the disutility of inconvenients to change. Resistance to change includes
inertia, frictions, obstacles, difficulties to change, costs to be able to
change and inconvenient to change;
\item [ii)] worthwhile transitions, i.e, succession of worthwhile single stays and
changes;
\item [iii)] traps, which can be stationary or variational. A trap is stationary
when, starting for it, no feasible change is worthwhile. An equilibrium
appears to be a very particular stationary trap, in a world with no
resistance to change, when only motivation to change matters. In this case
the agent has no motivation to change (no advantage to change, i.e only
losses to change), and zero resistance to change. A trap is variational with
respect to a subset of initial positions, when it is stationary, and,
starting from any of these initial positions, agents can find a succession
of worthwhile single changes and temporary stays which converge to this stationary
trap. Then, a variational trap is rather easy to reach and difficult to
leave in a worthwhile way. Furthermore, traps can be weak or strong,
depending of large or strict inequalities.
\end{itemize}
This last section, devoted to applications, focus the attention on a
succession of worthwhile changes and stays, moving from a weak stationary trap to a new one, given that the agent can change, each step, his satisficing
worthwhile to change ratio. The approximation process presented in~\cite{CLMM2012} represents a nice instance of such a worthwhile stability and change dynamic. The result of this paper shows that this dynamic converges to an equilibrium (a personal equilibrium for personal equilibrium problems), which is supposed, a priori, to be a weak variational trap.
This last hypothesis is the essence, in the variational rationality
approach, of the so called famous weak sharp minimum condition. It allows
convergence in finite time. Because, in the long run, we are dead, finite
convergence to a trap is a fundamental property of any human behavior
(defined as a sequence of actions). This worthwhile stability and change
dynamic is a very important benchmark case of the more general stability and
change dynamic (see \cite{S2009,S2010}), where a succession of worthwhile
changes move from a position to a new one (which are not supposed to be
stationary traps) and converges to an end point, which is shown to be a
variational trap (this is not an hypothesis). Finally, Section 6 contains
concluding discussions of the main results obtained in the paper.

\section{Preliminary}
\subsection{Riemanian Geometry}\label{sec2}
In this section, we recall some fundamental and basic concepts needed for reading this paper. These results and concepts can be found in the books on Riemannian geometry, see do Carmo~\cite{MP1992} and Sakay~\cite{S1996}.

Let $M$ be a $n$-dimensional connected manifold. We denote by $T_xM$ the $n$-dimensional {\it tangent space} of $M$ at $x$, by $TM=\cup_{x\in M}T_xM$ {\itshape{the tangent bundle}} of $M$ and by ${\cal X}(M)$ the space of smooth vector fields over $M$. When $M$ is endowed with a Riemannian metric $\langle .\,,\,. \rangle$, with the corresponding norm denoted by $\| . \|$, then $M$ is a Riemannian manifold. Recall that the metric can be used to define the length of piecewise smooth curves $\gamma:[a,b]\rightarrow M$ joining $x$ to $y$, i.e., $\gamma(a)=x$ and $\gamma(b)=y$, by $l(\gamma):=\int_a^b\|\gamma^{\prime}(t)\|dt$, 
and moreover, by minimizing this length functional over the set of all such curves, we obtain a Riemannian distance $d(x,y)$ inducing the original topology on $M$. We denote by $B(x,\epsilon)$ the Riemannian ball on $M$ with center $x$ and radius $\epsilon>0$. 
 A vector field $V$ along $\gamma$ is said to be {\it parallel} iff $\nabla_{\gamma^{\prime}} V=0$. If $\gamma^{\prime}$ itself is parallel we say that $\gamma$ is a {\it geodesic}. Given that the geodesic equation $\nabla_{\ \gamma^{\prime}} \gamma^{\prime}=0$ is a second-order nonlinear ordinary differential equation, we conclude that the geodesic \mbox{$\gamma=\gamma _{v}(.,x)$} is determined by its position $x$ and velocity $v$ at $x$. It is easy to verify that $\|\gamma ^{\prime}\|$ is constant. We say that $ \gamma $ is {\it normalized} iff $\| \gamma ^{\prime}\|=1$. The restriction of a geodesic to a closed bounded interval is called a {\it geodesic segment}.  Given points $x,y \in M$, we denote the geodesic segment from $x$ to $y$ by $[x,y]$. 
 We usually do not distinguish between a geodesic and its geodesic segment, as no confusion can arise. A geodesic segment joining $x$ to $y$ in $ M$ is said to be {\it minimal} iff its length equals $d(x,y)$ and the geodesic in question is said to be a {\it minimizing geodesic}. 

A Riemannian manifold is {\it complete} iff the geodesics are defined for any values of $t$. The \mbox{Hopf-Rinow's} Theorem~(\cite[Theorem 2.8, page 146]{MP1992} or~\cite[Theorem 1.1, page 84]{S1996}) asserts that, if this is the case, then any pair of points in $M$ can be joined by a (not necessarily unique) minimal geodesic segment. Moreover, $( M, d)$ is a complete metric space and, bounded and closed subsets are compact. From the completeness of the Riemannian manifold $M$, the {\it exponential map} $\exp_{x}:T_{x}  M \to M $ is defined by $\exp_{x}v\,=\, \gamma _{v}(1,x)$, for each $x\in M$.  A complete simply-connected Riemannian manifold of nonpositive sectional curvature is called an Hadamard manifold.  It is known that if $M$ is a Hadamard manifold, then $M$ has the same topology and differential structure as the Euclidean space $\mathbb{R}^n$; see, for instance,~\cite[Lemma 3.2, page 149]{MP1992} or \cite[Theorem 4.1, page 221]{S1996}. Furthermore, are known some similar geometrical properties to the existing in Euclidean space $\mathbb{R}^n$, such as, given two points there exists an unique geodesic segment that joins them. Now, we present a geometric property which will be very useful in the convergence analysis.

Let us recall that a {\it geodesic triangle \/} $\Delta (x_{1}x_{2}x_{3})$ of a Riemannian manifold is the set consisting of three distinct points $x_{1}$,$x_{2}$ ,$x_{3}$ called the {\it vertices \/}
and three minimizing geodesic segments $\gamma _{i+1}$ joining $x_{i+1}$ to $x_{i+2}$ called 
the {\it sides \/}, where $i=1,2,3 (mod\, 3) $.
\begin{theorem}\label{T:lawcos}
Let $M$ be a Hadamard manifold, $\Delta (x_{1}x_{2}x_{3})$ a geodesic triangle and 
$\gamma _{i+1}:[0,l_{i+1}]\to M $ geodesic segments joining $x_{i+1}$ to $x_{i+2}$ and set
$ l_{i+1}:=l(\gamma _{i+1}) $, 
$\theta _{i+1}= \, < \!\!\!)\,\left({\gamma '}_{i+1}(0),\, -{\gamma '}_{i}(l_{i})\right)$,
for $i=1,2,3$ $(mod\; 3)$. Then,
\begin{equation*} \label{E:sumangles}
\theta _{1} + \theta _{2} + \theta _{3} \leqslant \pi 
\end{equation*}
\begin{equation*} \label{E:lawcos}
l^{2}_{i+1}+l^{2}_{i+2}-2l_{i+1}l_{i+2}cos\theta _{i+2}\leqslant l^{2}_{i},
\end{equation*}
\begin{equation} \label{E:2lawcos}
d^2(x_{i+1},x_{i+2})+ d^2(x_{i+2}, x_{i})-2 \langle \exp^{-1}_{x_{i+2}}x_{i+1},\exp^{-1}_{x_{i+2}}x_{i}\rangle\leq d^2(x_{i},x_{i+1}).
\end{equation}
\end{theorem}
\begin{proof}
See, for example, \cite[Theorem 4.2, page 161]{S1996}.
\end{proof}

As mentioned in~\cite{CLMM2012} (see Bridson and Haefliger~\cite{BH1999}), it follows that:
\begin{equation}\label{ineq1}
d^2(x,y)\leq \langle \exp^{-1}_xz,\exp^{-1}_xy\rangle + \langle \exp^{-1}_yz,\exp^{-1}_yx\rangle,\qquad x, y, z\in M.
\end{equation}

{\it In this paper, all manifolds $M$ are assumed to be Hadamard and finite dimensional}.

\subsection{Convexity} \label{sec3}


A set $\Omega\subset M$ is said to be {\it convex}  iff any
geodesic segment with end points in $\Omega$ is contained in
$\Omega$, that is, iff $\gamma:[ a,b ]\to M$ is a geodesic such that $x =\gamma(a)\in \Omega$ and $y =\gamma(b)\in \Omega$, then $\gamma((1-t)a + tb)\in \Omega$ for all $t\in [0,1]$. Given $\mathcal{B}\subset M$, we denote by $\mbox{conv}(\mathcal{B})$ the convex {\it hull} of  $\mathcal{B}$, that is, the smallest convex subset of $M$ containing $\mathcal{B}$. Let $\Omega\subset M$ be a convex set. A function $f:\Omega\to\mathbb{R}$ is said to be {\it convex} iff for any geodesic segment $\gamma:[a, b]\to \Omega$ the composition $f\circ\gamma:[a,b]\to\mathbb{R}$ is convex. Take $p\in \Omega$. A vector $s \in T_pM$ is said
to be a {\it subgradient} of $f$ at $p$ iff
\begin{equation}\label{sgrad}
f(q) \geq f(p) + \langle s,\exp^{-1}_pq\rangle,\qquad q\in \Omega.
\end{equation}
The set of all subgradients of $f$ at $p$, denoted  by $\partial f(p)$, is called the { \it subdifferential} of $f$ at $p$. It is known that if $f$ is convex and $M$ is an Hadamard manifold, then $\partial f(p)$ is a nonempty set, for each $p\in \Omega$; see Udriste~\cite[Theorem 4.5, page 74]{U1994}. 

Let $\mathcal{B}\subset M$ be a non-empty, convex and closed  set. The distance function associated with $\mathcal{B}$ is given by 
\[
M\ni x \longmapsto d_{\mathcal{B}}(x):=\inf\{d(y,x): y\in \mathcal{B}\}\in \mathbb{R}_+.
\]
It is well-known (see~\cite[Corollary 3.1]{FO2002}) that for each $x\in M$ there exists a unique element $\tilde{x}\in \mathcal{B}$ such that 
\[
\langle \exp^{-1}_{\tilde{x}}x,\exp^{-1}_{\tilde{x}}y\rangle\leq 0,\qquad y\in \mathcal{B}.
\]
In this case, $\tilde{x}$ is the projection of $x$ onto the set $\mathcal{B}$ which we will denote by $P_{\mathcal{B}}(x)$.


\begin{remark}\label{rem2} 
It is important to mention that for every $y \in M$, $x\mapsto d(x,y)$ is a continuous and convex function, see~\emph{\cite[Proposition 4.3, page 222]{S1996}}. 
\end{remark}

\section{Equilibrium Problem}\label{sec3}

In this section, following the ideas shown in~\cite{IKS2009}, we present a sufficient condition for the existence of solution of equilibrium problems on Hadamard manifolds. We chose to present a proof only for the main result. With the exception to the proof of Proposition~\ref{prop1}, the proof of the other results can be extended, from those presented in linear environments (see \cite{IKS2009,IS2003}), with minor adjustments to the nonlinear context of this paper.

From now on, $\Omega \subset M$ will denote a nonempty closed convex set, unless explicitly stated otherwise. Given a bifunction $F: \Omega\times \Omega \rightarrow \mathbb{R}$ satisfying the property $F(x,x) =0$, for all $x \in \Omega$, the {\it equilibrium problem} in the Riemannian context (denoted by EP) consists in:
\begin{equation}\label{ep}
\mbox{Find $x^{*}\in \Omega$}:\quad F(x^\ast,y) \geq 0,\qquad \forall \ y \in \Omega.
\end{equation}
In this case, the bifunction $F$ is called an {\it equilibrium bifunction}. 
As far as we know, this problem was considered  firstly, in this context in~\cite{CLMM2012}, where the authors pointed out important problems, which are retrieved from  \eqref{ep}. Particularly, given $V\in \mathcal{X}(M)$, if 
\begin{equation}\label{eq:VIP}
F(x,y)= \langle V(x), \exp^{-1}_{x}y \rangle,\qquad \forall \ x, y\in \Omega,
\end{equation}
\eqref{ep} reduces to the variational inequality problem; see, for instance~\cite{N2003}. 
\begin{remark}
Although  the Variational Inequality Theory provides us a toll for formulating a variety of equilibrium problems, Iusem and Sosa ~\emph{\cite[Proposition 2.6]{IS2003}} showed that the generalization given by EP formulation  with respect to  VIP (Variational Inequality Problem) is genuina, in the sense that there are EP problems which do not fit the format of VIP. We affirm that is possible to guarantee the genuineness given by EP formulation compared to VIP, by considering the important class given by Quasi-Convex Optimization Problems that appear, for instance, in many micro-economical models devoted to maximize utilities. Indeed, the absence of convexity allows us to obtain situations where this important class of problems can not be considered as a VIP in the sense that their possible representation in this format would lead us to a problem, whose solution set contains points that do not necessarily belong to the solution set of the original optimization problem. On the other hand even in the absence of convexity, this class of problems  can be placed on the EP format.
\end{remark}

\begin{definition}
Let $F: \Omega\times \Omega \rightarrow \mathbb{R}$ be a bifunction. $F$ is said to be
\begin{itemize}
\item[\emph{(1)}] {\it monotone} iff  $F(x,y)+ F(y,x)\leq 0$, for all $(x,y)\in \Omega\times \Omega$;
\item[\emph{(2)}] {\it pseudomonotone} iff,  for each $(x,y)\in \Omega\times \Omega$, $F(x,y)\geq 0$ implies $F(y,x)\leq 0$;
\item[\emph{(3)}] {\it $\theta$-undermonotone} iff, there exists $\theta\geq 0$ such that, $F(x,y)+ F(y,x)\leq \theta d^2(x,y)$,  for all $(x,y)\in \Omega\times \Omega$.
\end{itemize}   
\end{definition}

\begin{remark}\label{rem31}$\;$
\begin{itemize}
\item [i)] Clearly, monotonicity implies pseudomonotonicity, but the converse does not hold even in a linear context, see, for instance, Iusem and Sosa~\emph{\cite{IS2003}}.
\item [ii)] If $F$ is pseudomonotone and, for $\tilde{x},\tilde{y}\in\Omega$,   $F(\tilde{x},\tilde{y})>0$ implies $F(\tilde{y},\tilde{x})<0$. Indeed, let us suppose, for contradiction, that $F(\tilde{y},\tilde{x})=0$ (in particular $F(\tilde{y},\tilde{x})\geq 0$). From the pseudomonotonicity of $F$ it follows that $F(\tilde{x},\tilde{y})\leq 0$, which is an absurdity, and the affirmation is proved.
\end{itemize}
\end{remark}


Next result was presented by Colao et al. in~\cite{CLMM2012} and is fundamental to establish our existence result for the EP.
\begin{proposition}\label{prop1}
Let $\mathcal{B}\subset M$ be  a closed convex subset and  $H: \mathcal{B}\to 2^{\mathcal{B}}$ be a mapping such that, for each $y \in \mathcal{B}$, $H(y)$ is closed. Suppose that
\begin{itemize}
\item[\emph{(i)}] there exists $y_0 \in \mathcal{B}$ such that $H(y_0)$ is compact;
\item[\emph{(ii)}] $\forall  y_1,\ldots, y_m \in \mathcal{B}$, $\emph{conv} ( \{ y_1,\ldots, y_m\})\subset\bigcup_{i=1}^mH(y_i)$.
\end{itemize}
Then, \[\bigcap_{y\in \mathcal{B}} H(y)\neq \varnothing.\]
\end{proposition}
\begin{proof}
See~\cite{CLMM2012}.
\end{proof}
Unless stated to the contrary, in the remainder of this paper we assume that $F:\Omega\times \Omega\to\mathbb{R}$ is a  bifuntion satisfying  the following assumptions: 
\begin{itemize}
\item[$\mathcal{H}$1)] $F(x,x)=0$ for each $x\in \Omega$;
\item[$\mathcal{H}$2)] For every $x\in \Omega$,  $y\mapsto F(x,y)$ is convex and lower semicontinuous;
\item[$\mathcal{H}$3)] For every $y\in \Omega$,  $x\mapsto F(x,y)$ is upper semicontinuous.
\end{itemize}


For each $y\in \Omega$, let us define:
\[
L_F (y) := \{x\in \Omega : F(y, x)\leq 0\}.
\]
From this set, we can consider the following {\it convex feasibility problem} (denoted by CFP):
\[
\mbox{Find $x^{*}$}\in \bigcap_{y\in \Omega} L_F(y).
\]
As far as we know, this problem was first studied, in the Riemannian context, by Bento and Melo in~\cite{Bento2012}, in the particular case where the domain of $F$ is $M\times \{1,\ldots, m\}$. In this case, $y\in\{1,\ldots, m\}$ and $\Omega$ is the whole $\mbox{M}$. 


Next result establishes a relationship between CFP and EP.
\begin{lemma}\label{rem3}
The solution set of \emph{CFP} is contained in the solution set of \emph{EP}.
\end{lemma}

\begin{remark}
Note that, as it is in the Euclidean context, the equality between the sets in the previous lemma in general does not happens, see~\cite{IS2003}. However, in the particular case where $F$ is pseudomonotone, the equality is immediately verified.
\end{remark}

Take $z_0\in M$ fixed.  For each $k\in \mathbb{N}$ consider the following set: 
\[
\Omega_k:=\{x\in \Omega:\; d(x,z_0)\leq k\}. 
\]
Note that $\Omega_k$ is a nonempty set, for $k\in \mathbb{N}$ sufficiently large. For simplicity, we can suppose, without loss of generality, that $\Omega_k$ is a nonempty set for all $k\in\mathbb{N}$. Moreover, as $\Omega_{k}$ is contained in the closed ball $B(z_0,k):=\{x\in M:\; d(z_0,x)\leq k\}$, it is a bounded set.
On the other hand, since $d(\cdot, z_{0})$ is a continuous and convex function (this follows from Remark \ref{rem2}), $\Omega_k$ is a convex and closed set and, hence, compact (see the Ropf-Rinow's Theorem). We denote, by $\Omega_k^0$, the following set: 
\[
\Omega_k^0:=\{ x\in \Omega : d(x,z_0)< k\}.
\] 
For each $y\in \Omega$, let us define:
\[
L_F (k, y) := \{x\in \Omega _k : F(y, x)\leq 0\}.
\]

\begin{lemma}\label{prop2}
Let $k\in \mathbb{N}$, $\bar{x}\in \bigcap_{y\in \Omega_k} L_F (k, y)$ and assume that there exists $\bar{y} \in \Omega_k^0$ such that  $F (\bar{x}, \bar{y}) \leq 0$. Then, $F(\bar{x}, y)\geq 0$, for all $y\in \Omega$, i.e., $\bar{x}$ is a solution for \eqref{ep}. 
\end{lemma}

\begin{assumption}\label{assumption-13-1}
Given $k\in \mathbb{N}$, for all finite set $\{y_1,\ldots, y_m \}\subset \Omega_{k}$, one has 
\[
\emph{conv} ( \{ y_1,\ldots, y_m\})\subset\bigcup_{i=1}^mL_{F}(k,y_i).
\]
\end{assumption}

\begin{remark}\label{lem1} 
Note that, in the particular case where F is pseudomonotone, the property described by the previous assumption is naturally verified. Indeed, let $y_1,\ldots, y_m\in \Omega_k$,  take $\bar{y}\in \emph{conv}(\{y_{1},\ldots, y_{n}\})$ and let us suppose, for contradiction, that $\bar{y}\notin \bigcup_{i=1}^{m}L_{F}(k,y_i)$. Then,
  \begin{equation}\label{eq:equil1}
  F(y_{i},\bar{y})>0, \qquad i\in \{1,\ldots, m\}.
  \end{equation}
 Now, define the following set $B:=\{x\in \Omega_{k}:\; F(\bar{y}, x)<0\}$. In the particular case where $F$ is \mbox{pseudomonotone}, using \eqref{eq:equil1} and taking into account that $B$ is convex (this follows from $\mathcal{H}2$), we conclude that $\bar{y}\in B$ (see item $ii)$ of Remark~\ref{rem31}). But this contradicts $\mathcal{H}1$ and the affirmation  is proved.
  \end{remark}   
  
\begin{assumption}\label{assumption-13-2}
Given $z_{0}\in M$ fixed, consider a sequence  $\{z^k\} \subset \Omega$ such that $\{d(z^k,z_0)\}$ converges to infinity as $k$ goes to infinity.
Then, there exists $x^\ast \in \Omega$ and $k_0 \in \mathbb{N}$ such that 
\begin{equation*}\label{ineq:2013-3}
F(z^k, x^\ast)\leq 0, \qquad k\geq k_0.
\end{equation*}
\end{assumption}

It is worth noting that this last assumption has been presented by Iusem et al.~\cite{IKS2009}, in a space with a linear structure. It is a sufficient condition for the existence of solutions of the equilibrium problem EP. 

Next result (see~\cite{IS2003} for similar results, in the linear setting) assure us that Assumption~\ref{assumption-13-2} is a weaker sufficient condition than the coercivity assumption used by Colao et al.~\cite{CLMM2012}, for the existence of solutions of EP.

\begin{proposition}\label{prop3} 
Let  $\mathcal{B}\subset M$ be a compact set and  $y_0 \in  \mathcal{B}\cap \Omega$ a point such that
$F(x, y_0)< 0$, for all $x \in  \Omega \setminus \mathcal{B}$. Then, $F$ satisfies \emph{Assumption~\ref{assumption-13-2}}.
\end{proposition}

The following is the main result of this section.
\begin{theorem}\label{teo1}
Under \emph{Assumptions~\ref{assumption-13-1}} and \emph{\ref{assumption-13-2}},  EP admits a solution. 
\end{theorem} 
\begin{proof}
Recall that $\Omega_{k}$ is a convex and compact set for each $k\in \mathbb{N}$. Now, given $k\in \mathbb{N}$ and $y\in \Omega$, note that $L_F (k, y)$ is a compact set. Indeed, this fact follows from the definition of $L_{F}(k,y)$ combined with \mbox{assumption} $\mathcal{H}2$ ($F(y,\cdot)$ is a lower semicontinuous function on $\Omega$) and compactness of $\Omega_{k}$. Now, since \mbox{Assumption} \ref{assumption-13-1} holds true, using Proposition~\ref{prop1} with $\mathcal{B}=\Omega_{k}$ and $H(y)=L_{F}(k,y)$, we conclude that, for each $k\in \mathbb{N}$,
\[
\bigcap_{y\in \Omega_k}  L_F (k, y)\neq \emptyset.
\]
For each $k$, choose $z^k \in \bigcap_{y\in \Omega_k} L_F (k, y)$ and take $z_{0}\in M$ fixed. If  there exists $k\in \mathbb{N}$ such that $d(z^k,z_0) < k$, then $z^k \in\Omega_k^0$ and, from Lemma~\ref{prop2}, it follows that $z^{k}$ solves EP. On the other hand, if $d(z^k,z_0)= k$, from Assumption~\ref{assumption-13-2}, there exists,  $x^\ast \in \Omega$ and $k_{0}\in\mathbb{N}$ such that  $F (z^k, x^\ast)\leq 0$, for all $k\geq k_0$. Taking $ k'> k_0$ such that $d(x^\ast,z_0)<k'$, we have $F (z^{k'},x^\ast)\leq 0$ and $x^\ast \in \Omega _{k'}^0$. Therefore, using again Lemma~\ref{prop2},  we conclude that $z^{k'}$  solves EP, and the proof is complete. 
 \end{proof}

Next example was inspired by~\cite[Example 3.4]{CLMM2012}. It illustrates the usefulness of the our previous result, in the sense that it applies to some situations not covered in the linear setting. For other papers that highlight such advantage, in regard to the linear setting, see \cite{CFPN2006, Bento2012}.  
\begin{example}
Let $\Omega=\{(x,y,z): 0\leq x \leq 1, y^2-z^2=-1, y\geq 0, z\geq 1\}\subset \mathbb{R}\times \mathbb{H}^1$
and consider the \mbox{following} bifunction $F:\Omega\times \Omega \rightarrow \mathbb{R}$, given by:
\[
F((x_1, y_1, z_1), (x_2, y_2, z_2)):=(2- x_1)\left(\left(y_2^2+z_2^2\right)-\left(y_1^2+z_1^2\right)\right).
\]
Note that $\Omega$ is indeed a not convex set in $\mathbb{R}^3$. So,  an equilibrium problem defined on $\Omega$  cannot be solved by using the classical results known in the linear context.  Let $(\mathbb{H}^{n}, \langle \, , \, \rangle)$ be the
Riemannian manifold, where
\[
\mathbb{H}^{n}:=\{ x=(x_1, x_2, \ldots, x_{n+1})\in \mathbb{R}^{n+1}: x_{n+1} >0 \; \mbox{and} \; \langle x,x\rangle=-1 \} \quad \mbox{(hyperbolic n space)},
\] and 
$\langle \, , \,
\rangle
$
 is the Riemannian metric  $\langle x,y \rangle:= x_1y_1+x_2y_2+\ldots +x_{n}y_{n} -x_{n+1}y_{n+1}$ (Lorentz metric). As noted in \emph{\cite{CLMM2012}}, $(\mathbb{H}^{n}, \langle \, , \, \rangle)$ is a Hadamard manifold  with sectional curvature $-1$ and, given initial conditions $x\in \mathbb{H}^n$, $v \in T_x \mathbb{H}^n$ ($\|v\|=1$), the  normalized geodesic $\gamma: \mathbb{R}\to\mathbb{H}^n$, is given by:
\[
\gamma (t) = (\cosh t)x + (\sinh t)v, \qquad t \in \mathbb{R}.
\]
Hence, we obtain the following expression for the Riemannian distance $d$:
\[
d(x,y)=\emph{arccosh}(-\langle x,y\rangle),\qquad  x,y\in \mathbb{H}^n.
\]
Again, as observed in \emph{\cite{CLMM2012}},  $\Omega$ is a convex set which is immersed in the Hadamard manifold $M:=\mathbb{R}\times \mathbb{H}^1$. Using the expression of the geodesic curves, it can be deduced that $F$ is a convex function in the second variable. Moreover, from the definition of $F$, it is easy to see  that all the assumptions $\mathcal{H}1$, $\mathcal{H}2$ and $\mathcal{H}3$ are satisfied, and $F$ is a pseudomonotone bifunction which is not monotone. In particular, from \emph{Remark~\ref{lem1}}, it follows that \emph{Assumption \ref{assumption-13-1}} holds. Now, take $w^0\in M$ fixed and a sequence $\{w^{k}\}\subset \Omega$, $w^k:=(x_k,y_k, z_k)$, such that $d(w^k, w^0)\to+\infty$. There exists  $x^\ast:=(1, 0, 1)\in \Omega$ such that $F(w^k,x^\ast)\leq 0$, for all $k\in \mathbb{N}$, i.e., \emph{Assumption \ref{assumption-13-2}} holds. Therefore, \emph{Theorem \ref{teo1}} implies the existence of an equilibrium point for $F$.

\end{example}


\section{Proximal Point for Equilibrium Problem}\label{sec4}
In this section and remainder of this paper, $M$ will denote an Hadamard manifold with null sectional curvature. Following some ideas presented in~\cite{IS2010},  we propose an approach of  the proximal point algorithm for equilibrium problems on Hadamard manifolds with null sectional curvature, where the convergence result is obtained for bifunctions which are not necessarily monotone. This problem was proposed in~\cite{CLMM2012} for the case of an Hadamard manifold and under monotonocity of the equilibrium bifunction. 

Let us denote the equilibrium point set of $F$ by EP(F,$\Omega$) and, for $\lambda>0$ and $z\in \Omega$ fixed, consider the bifunction 
\begin{equation}\label{eq:regularization1}
F_{\lambda, z}(x,y):=F(x,y)-\lambda\langle \exp^{-1}_{x}z,\exp^{-1}_xy\rangle, \qquad x,y\in\Omega.
\end{equation}
Now, we describe a proximal point algorithm  to solve the equilibrium problem \eqref{ep}.
\begin{algorithm}\label{alg1} Take $\{\lambda_{k}\}$ a bounded sequence of positive real numbers.
\begin{itemize}
\item [1.] Choose an initial point $x^{0}\in \Omega$;
\item [2.] Given $x^{k}$,  if $x^{k}\in \emph{EP}(F,\Omega)$, STOP. Otherwise;
\item [3.] Given $x^{k}$, take as the next iterate any  $x^{k+1}\in\Omega$ such that:
\begin{equation}\label{mpp}
x^{k+1}\in \emph{EP}(F_{k},\Omega), \quad F_{k}:=F_{\lambda_{k}, x^{k}}.
\end{equation}
\end{itemize}
\end{algorithm}
\begin{remark}
It is worth noting that the iterative process \eqref{mpp} has appeared first in~\emph{\cite{CLMM2012}}. If $V\in\mathcal{X}(M)$, note that, for $F$ given as in \eqref{eq:VIP}, this iterative process retrieves the proximal point method for the \mbox{variational} \mbox{inequalities} problem on Hadamard manifolds presented in~\emph{\cite{Tang2013}}. In particular, the \mbox{iterative} process \eqref{mpp} retrieves the proximal point method for minimization problems, see~\emph{\cite{FO2002}} or, more generally, the proximal point for vector fields both on Hadamard manifolds, see~\emph{\cite{LLMM2009}}.
\end{remark}
Next results are useful to ensure the well-definition of Algorithm \ref{alg1}. In the remainder of this section, we assume that $\lambda$ is a positive real number and 
$z \in \Omega$, both fixed.
\begin{lemma}\label{lem2} 
Let $F$ be a $\theta$-undermonotone bifunction with $\theta\leq \lambda$. Then, $F_{\lambda, z}$ is monotone.
\end{lemma}
\begin{proof} 
From \eqref{eq:regularization1}, it is easy to see that 
\[
F_{\lambda, z}(x,y)+F_{\lambda, z}(y,x)=F(x,y)+F(y,x)-\lambda[\langle \exp^{-1}_xz,\exp^{-1}_xy\rangle + \langle \exp^{-1}_yz,\exp^{-1}_yx\rangle ], \qquad x, y\in\Omega.
\]
So, taking into account that $F$ is $\theta$-undermonotone, the desired result follows by combining last equality with \eqref{ineq1} and assumption $\theta\leq \lambda$.
\end{proof}

\begin{lemma}\label{lem:06-11}
$F_{\lambda, z}$ satisfies the assumption $\mathcal{H}$2.
\end{lemma}
\begin{proof}
From the definition of $F_{\lambda, z}$ in \eqref{eq:regularization1} and, taking into account that $\lambda>0$ and $F$ satisfies $\mathcal{H}$1, to prove this lemma it is sufficient to ensure that,  $\Omega\ni y\longmapsto g(y):=-\langle \exp^{-1}_{x}z,\exp^{-1}_{x}y  \rangle\in\mathbb{R}$ is convex and lower semicontinuous. For the convexity of $g$, see Cruz Neto et al.~\cite[Theorem 1]{CMSS2014}.  Note that $g$ is a lower semicontinuous  function, since $g$ is a differentiable function. 
\end{proof}
\begin{lemma}\label{lem3} 
Let $F$ be a $\theta$-undermonotone bifunction with $\theta<\lambda$. If $F$ satisfies \emph{Assumption \ref{assumption-13-2}}, then $F_{\lambda, z}$ also satisfies this assumption. 
\end{lemma}
\begin{proof} 
First of all, given $z_{0}\in M$, consider a sequence $\{z^k\}\subset\Omega$ such that $\{d(z^k,z_0)\}$ converges to infinity as $k$ goes to infinity. 
Using \eqref{eq:regularization1} with $x=z^{k}$ and $y=z$, we get 
\begin{eqnarray}\label{eq1}
F_{\lambda,z}(z^k, z)&=& F(z^k,z)-\lambda \langle \exp^{-1}_{z^k}z,\exp^{-1}_{z^k}z\rangle,\nonumber\\
&=& F(z^k,z)-\lambda d^2(z^k,z),\nonumber\\
&\leq& -F(z,z^k)+(\theta-\lambda)d^2(z^k,z),
\end{eqnarray}
where the last inequality follows from the $\theta$-undermonotonicity of $F$. Let us show that $F_{\lambda,z}(z^k,z)\leq 0$, for all $k\geq k_0$. Define the function $f_z:\Omega\rightarrow\mathbb{R}$ by $f_z(y)=F(z,y)$.
Since $M$ is an Hadamard manifold and $y\mapsto F(z,y)$ is a convex function, there exists $v'\in \partial f_z(x')$. So, applying inequality \eqref{sgrad} with $f=f_{z}$, $s=v'$, $p=x'$ and $q=z^{k}$, we have  
\begin{equation}\label{eq2}
\langle v' ,\exp^{-1}_{x'}z^k\rangle \leq f_z(z^k)-f_z(x')= F(z,z^k)- F(z,x'),\qquad k=0,1,\dots.
\end{equation}
From \eqref{eq2} and Cauchy-Schwarz inequality, 
\begin{equation}\label{eq3}
-F(z,z^k)\leq \| v'\|d(z^k,x')- F(z,x')\leq \| v'\|[d(z^k,z) +d(z,x')]- F(z,x'),\qquad k=0,1,\ldots.
\end{equation}
Using \eqref{eq1} and \eqref{eq3}, we have
\begin{eqnarray}\label{ineq:welldef1}
F_{\lambda,z}(z^k,z) &\leq& \| v'\|[d(z^k,z) +d(z,x')]- F(z,x')+ (\theta-\lambda)d^2(z^k,z), \qquad k=0,1,\ldots\nonumber\\
&=&  d(z^k,z)[\| v'\|+(\theta-\lambda)d(z^k,z)]+ \| v'\|d(z,x')- F(z,x'), \qquad k=0,1,\ldots.
\end{eqnarray}
Now, taking into account that $\theta<\lambda$ and $\{d(z^k,z_0)\}$ converges to infinity as $k$ goes to infinity, letting $k$ goes to infinity,  we obtain  $(\theta-\lambda)d(z^k,z)\to -\infty$. Hence, the desired results follows from the inequality \eqref{ineq:welldef1} which concludes the proof. 
 \end{proof}

\begin{theorem}\label{teo2}
Assume that \emph{Assumption~\ref{assumption-13-2}} holds and  $F$ is a $\theta$-undermonotone bifunction with $\theta<\lambda$. Then, there exists an unique $\bar{x}^\ast \in\Omega$ such that
\[
F_{\lambda, z}(\bar{x}^\ast, y)\geq 0,\qquad y\in \Omega.
\]
\end{theorem}
\begin{proof}
From Lemma \ref{lem2} it follows that $F_{\lambda,z}$ is monotone and, in particular, pseudomotone (this follows from Remarks \ref{rem31}).  Moreover, Remark~\ref{lem1} implies that $F_{\lambda, z}$ satisfies {Assumption \ref{assumption-13-1}} and Lemma~\ref{lem:06-11} (resp. Lemma~\ref{lem3}) tell us that $F_{\lambda, z}$ satisfies $\mathcal{H}2$ (resp. Assumption \ref{assumption-13-2}). Hence, from Theorem~\ref{teo1} there exists, a point $\bar{x}_{1}^\ast \in \Omega$ such that
\[
F_{\lambda, z}(\bar{x}^\ast, y)\geq 0,\qquad y\in \Omega.
\]
Let us suppose, by contradiction,  that there exists $\bar{x}_{2}^\ast$ satisfying the last inequality. Then,  
\[
F_{\lambda,z}(\bar{x}_1^\ast,\bar{x}_2^\ast)=F(\bar{x}_1^\ast,\bar{x}_2^\ast)-\lambda \langle \exp^{-1}_{\bar{x}_1^\ast} z,\exp^{-1}_{\bar{x}_1^\ast}\bar{x}_2^\ast\rangle \geq 0,
\]
and
\[
F_{\lambda, z}(\bar{x}_2^\ast,\bar{x}_1^\ast)=F(\bar{x}_2^\ast,\bar{x}_1^\ast)-\lambda \langle \exp^{-1}_{\bar{x}_2^\ast} z,\exp^{-1}_{\bar{x}_2^\ast}\bar{x}_1^\ast\rangle \geq 0.
\] 
By summing the last two inequalities, we get
\[
\lambda [\langle \exp^{-1}_{\bar{x}_1^\ast} z,\exp^{-1}_{\bar{x}_1^\ast}\bar{x}_2^\ast\rangle  + \langle \exp^{-1}_{\bar{x}_2^\ast} z,\exp^{-1}_{\bar{x}_2^\ast}\bar{x}_1^\ast\rangle]\leq F(\bar{x}_1^\ast,\bar{x}_2^\ast)+F(\bar{x}_2^\ast,\bar{x}_1^\ast)  
\leq \theta d^2(\bar{x}_1^\ast,\bar{x}_2^\ast)< \lambda d^2(\bar{x}_1^\ast,\bar{x}_2^\ast),
\]
which contradicts inequality \eqref{ineq1} and the proof is concluded.
 \end{proof}

\begin{corollary}
Assume that \emph{Assumption~\ref{assumption-13-2}} holds and  $F$ is a $\theta$-undermonotone bifunction. If $\{\lambda_{k}\}$ is a  bounded sequence of positive real numbers such that $\theta<\lambda_{k}$, $k\in\mathbb{N}$, then Algorithm~\ref{alg1} is well-defined.
\end{corollary}

\begin{proof}
It follows immediately from Theorem~\ref{teo2}.
 \end{proof}
In the remainder of this paper we assume that the assumptions of the previous corollary hold and $\{x^k\}$ is a sequence generated from Algorithm~\ref{alg1}. Taking into account that if  Algorithm~\ref{alg1} terminates after a finite number of iterations, it terminates at an equilibrium point of $F$, from now on, we assume also that  $\{x^k\}$ is an infinite sequence.

\subsection{Convergence Analysis}

In this section we present the convergence of the sequence $\{x^{k}\}$.     

\begin{proposition}
Let $F$ be a pseudomonotone bifunction. Then, 
\begin{equation}\label{ineq3}
\langle \exp^{-1}_{x_{k+1}}x^k,\exp^{-1}_{x_{k+1}}x^\ast\rangle \leq 0, \qquad x^\ast \in EP(F,\Omega).
\end{equation}
\end{proposition}
\begin{proof}
From the definition of the iterate $x^{k+1}$ and $F_{k}$ in \eqref{mpp} combined with \eqref{eq:regularization1}, we obtain
\[
\lambda_k \langle \exp^{-1}_{x_{k+1}}x^k,\exp^{-1}_{x_{k+1}}y\rangle \leq F(x^{k+1},y),\qquad y \in\Omega.
\] 
Since $F$ is pseudomonotone and  $EP(F,\Omega)\subset\Omega$, the desired result follows from the last inequality.
 \end{proof}

\begin{definition}\rm A sequence $\{z^k\}$ in the complete metric
space $(M,d)$ is said to be {\it Fej\'er convergent\/} to a nonempty set $\mathcal{S}\subset M$ iff for every
$z \in \mathcal{S}$,
\[
d(z^{k+1},z) \le d(z^k,z)\qquad k=0,1,\ldots.
\]
\end{definition}

The following result is well known and its proof is elementary.

\begin{proposition}\label{prop4}
Let $\{z^k\}$ be a sequence in the complete metric
space $(M,d)$. If $\{z^k\}$ is Fej\'er convergent to a non-empty set
$\mathcal{S} \subset M$, then $\{z^k\}$ is bounded. If, furthermore, an accumulation
point $z$ of $\{z^k\}$ belongs to $\mathcal{S}$, then
$\displaystyle \lim_{k\to\infty} z^k=z$.
\end{proposition} 


Now, we present our main convergence result.

\begin{theorem}\label{teo3}
Assume that $F$ is pseudomonotone. The sequence  $\{x^k\}$ converges to a point in \emph{EP}(F, $\Omega$). 
\end{theorem}
\begin{proof}
Take $\bar{x}\in \mbox{EP}(F, \Omega)$. Using inequality \eqref{E:2lawcos} with $x_{i}=\bar{x}$, $x_{i+1}=x^{k}$ and $x_{i+2}=x^{k+1}$, we obtain 
\[
d^2(x^{k+1},\bar{x})+ d^2(x^{k+1}, x^k)-2 \langle \exp^{-1}_{x^{k+1}}x^k,\exp^{-1}_{x^{k+1}}\bar{x}\rangle\leq d^2(x^{k},\bar{x}). 
\]
Since $F$ is pseudomonotone and $x^{*}\in \mbox{EP}(F, \Omega)$,  combining  inequality \eqref {ineq3} with last inequality and taking into account that $d^2(x^{k+1}, x^k)>0$, it follows that $\{x^{k}\}$ is Fej\'er convergent to the set  $\mbox{EP}(F, \Omega)$. So, applying Proposition~\ref{prop4} with $z^{k}=x^{k}$, $k\in\mathbb{N}$,  and $\mathcal{S}=\mbox{EP}(F, \Omega)$, we have that $\{x^{k}\}$ is a bounded sequence. In particular,  from the Hopf-Rinow Theorem, there exists a subsequence $\{x^{k_{j}}\}$ of $\{x^{k}\}$ converging to some point $x^{*}$. Besides, as $\{\lambda_{k}\}$ also is a bounded sequence, without loss of generality we can suppose that $\{\lambda_{k_{j}}\}$ is a subsequence of $\{\lambda_{k}\}$ converging to some $\lambda_{*}$. Given $y\in \Omega$ and considering that the angle between the vectors $\exp^{-1}_{x_{k_{j}+1}}x^{k_{j}}$, $\exp^{-1}_{x_{k_{j}+1}}y$ is denoted by $\theta_{k_{j}}=\, < \!\!\!)\,\left(\exp^{-1}_{x_{k_{j}+1}}x^{k_{j}},\, \exp^{-1}_{x_{k_{j}+1}}y\right)$,  taking into account that $\|\exp^{-1}_{x}\tilde{x}\|=d(x,\tilde{x})$, for all $x,\tilde{x}\in M$,
we obtain
\begin{equation}\label{eq:cos10}
\langle \exp^{-1}_{x_{k_{j}+1}}x^{k_{j}},\exp^{-1}_{x_{k_{j}+1}}y\rangle =d(x^{k_{j}+1}, x^{k_{j}})d(x^{k_{j}+1},y)\cos\theta_{k_{j}}, \qquad j\in\mathbb{N}.
\end{equation}
Now, from the definition of $x^{k_{j}+1}$ in \eqref{mpp} combined with  \eqref{eq:regularization1} and \eqref{eq:cos10}, we get
\begin{equation}\label{ineq:conv10}
F(x^{k_j+1},y)-\lambda_{k_{j}} d(x^{k_{j}+1}, x^{k_{j}})d(x^{k_{j}+1},y)\cos\theta_{k_{j}}\geq 0, \qquad j\in\mathbb{N}.
\end{equation}
Since $F(\cdot\, , y)$ is upper semicontinuous, $\{\cos \theta_{k_{j}}\}$ is bounded and $\{d(x^{k_j},x^{k_j+1})\}$ (resp. $\{\lambda_{k_{j}}\}$) goes to zero (resp. $\lambda_{*}$) as $j$ goes to infinity, we have
\[
F(x^\ast,y)\geq \lim F(x^{k_j+1},y)= \lim_{j\to+\infty}[F(x^{k_j+1},y)-\lambda_k \langle \exp^{-1}_{x^{k_j+1}} x^{k_j}, \exp^{-1}_{x^{k_j+1}} y\rangle]\geq 0,
\]
where the last inequality follows from \eqref{ineq:conv10}. Therefore, the desired result follows of the arbitrary of $y\in \Omega$.
\end{proof}

Following the notations presented in~\cite{CLMM2012}, let us consider, for each $\lambda>0$ fixed,  the following set-valued operator $J^F_{\lambda} :M \to 2^{\Omega}$ given by
\begin{equation}\label{eq:resolventColao}
J^F_{\lambda}(x):=\{z\in M:\; F(z,y)-\lambda\langle \exp^{-1}_{z}{x},\exp^{-1}_{z}y\rangle\geq 0,\quad y\in \Omega\}.
\end{equation}
In the particular case where $F$ is monotone, $J_{\lambda}^{F}$ is a firmly nonexpansive operator, i.e., for any $x, y\in \Omega$, the function $\phi:[ 0,1 ]\rightarrow [0, \infty]$, defined by $\phi(t) = d (\gamma_1(t),\gamma_2(t))$, is nonincreasing, where $\gamma_1(t)$ and $\gamma_2(t)$ denote the geodesics joining $x$ to $J_{\lambda}^{F}(x)$ and $y$ to $J_{\lambda}^{F}(y)$, respectively. The proof of convergence presented in ~\cite{CLMM2012} is based exactly on this property.  Now, we present an example just for  illustrating that, as it is in the Euclidean context, the assumed conditions in our convergence result are weak to ensure the mentioned  property. This is one of the reasons why, in the present paper, do not we  followed the approach presented in~\cite{CLMM2012} to obtain Theorem~\ref{teo3}.


\begin{example}\label{ex1}\emph{(see~\cite{IS2010})} Let us consider $M=\mathbb{R}$, $\Omega=[1/2,1]$ and $F:[1/2,1]\times [1/2,1]\rightarrow \mathbb{R}$, given by: \[
F(x,y)=x(x-y).
\]
It is easy to check that $F$ is {\it $1$-undermonotone} and {\it pseudomonotone}. Now, we show that the resolvent of the bifunction $F$ is not {\it firmly nonexpansive}. Take $x^0=1/2$ and $\lambda =7$. It follows that  
\[
J^F_{\lambda}(x^0)=\frac{\lambda}{\lambda-1}x^0 \; \therefore \; x^{k}:=\left(J^F_{\lambda}(x^0)\right)^k= \min \left\lbrace 1, \left( \frac{\lambda}{\lambda-1}\right)^k x^0\right\rbrace.
\] 
Hence, $x^1=\lambda/2(\lambda-1)=7/12$ and $x^2=\lambda^2/2(\lambda-1)^2=49/72$. Now, denote $T(x)=J^F_{\lambda}(x)$, $x=x^0$, $y=T(x)=x^1$ and $T(y)=x^2$. So, 
\begin{eqnarray}
\langle \exp^{-1}_{T(x)}{T(y)},\exp^{-1}_{T(x)}x\rangle +\langle \exp^{-1}_{T(y)}{T(x)},\exp^{-1}_{T(y)}y \rangle 
  &=& (x^2-x^1)(x^0-x^1)+ (x^1-x^2)(x^1-x^2),  \nonumber \\
   &=&(x^2-x^1)(x^0-2x^1+x^2), \nonumber \\
      &=& \frac{1}{2}(x^2-x^1)\left(\frac{\lambda}{\lambda-1}-1\right)^2,\nonumber\\ 
      &>& 0.\nonumber  
\end{eqnarray}
Therefore, by~\emph{\cite[Proposition 5]{LLMW2011}} the resolvent $J^F_{\lambda}$ is not firmly nonexpansive.
\end{example}

\subsection{Finite Termination}

In this section, following the ideas presented in~\cite{M2007}, we obtain an important result of finite termination for any sequence generated from Algorithm~\ref{alg1}.


Next definition was introduced in the context of Hilbert spaces, see~\cite{M2007}.
\begin{definition}\label{condicionada}
A possibly non degenerated $F$  is said to be $\rho$-conditioned if and only if there exist positive numbers $\tau$ and $\rho$ such that
\begin{equation}\label{ineq5}
-F (x, P_S(x)) \geq \tau \emph{dist}^{\rho}(x, S), \qquad x\in \Omega,
\end{equation}
where $S=\emph{EP}(F,\Omega)$. 
\end{definition}
Note that, if $F$ is degenerate, i.e., exists a function $\phi :\Omega \to \mathbb{R}$ such that $F(x,y)=\phi (y)-\phi (x)$, then $S=argmin _{\Omega}\phi$.  In this particular case, 
if $\phi$ is convex and $\rho=1$, Definition~\ref{condicionada} reduces to 
\begin{equation}\label{ineq:weaksharp11-2013}
\phi(x)\geq \phi(\bar{x})+\alpha d_{S}(x),\quad \bar{x}\in S, \quad x\in\Omega.
\end{equation}
This notion has been introduced in~\cite{LMWY2011}, in the Riemannian context,  and says that $S$  is a weak sharp minima set for the minimization problem 
\begin{equation}\label{minoptprob1}
\min\{\phi(x):\;  x\in \Omega\},
\end{equation}
 with modulus $\alpha>0$. 

Now, we present a finite convergence result for the sequence $\{x^{k}\}$. 

\begin{theorem}\label{teo4}
Assume that S is a non-empty set, the inequality \eqref{ineq5} holds with $\rho\in ]0,1]$ and  the sequence $\{x^k\}$ converges to $x^\ast \in S$.  Then, $x^*$ is reached in a finite number of iterations, i.e., there exists a number $k_0\in \mathbb{N}$ such that 
\[
x^{k_0}=x^\ast.
\]
\end{theorem}
\begin{proof}
Assume that $dist(x^{k+1}, S)>0$ for all $k\in \mathbb{N}$. From \eqref{ineq5}, we obtain
\[
\tau \emph{dist}^{\rho}(x^{k+1}, S)\leq -F(x^{k+1},P_S(x^{k+1}))\leq - \lambda_k\langle\exp^{-1}_{x_{k+1}}x^k,\exp^{-1}_{x^{k+1}}P_S(x^{k+1}) \rangle \leq \bar{\lambda}d(x^k, x^{k+1})dist(x^{k+1},S),   
\]
for all $k\in \mathbb{N}$ and $\bar{\lambda}=\sup_{k\in \mathbb{N}}\{\lambda_k\}$. 
From the last inequality, we get    
\begin{equation*}\label{ineq6}
\bar{\lambda}^{-1}{\tau} \emph{dist}^{\rho-1}(x^{k+1}, S)\leq d(x^k, x^{k+1}).
\end{equation*}
Since $\rho-1\leq 0$ and $dist(x^{k+1},S)\to 0$, it follows that $d(x^k, x^{k+1})\nrightarrow 0$, 
which contradicts the fact that $\{x^k\}$ is a convergent sequence. Therefore, the result of the theorem is true. 
\end{proof}

\begin{remark}$\;$
\begin{itemize}
\item[i)] Note that the last theorem extend  the finite termination result in~\cite{BC2012} to minimization problems whose constrained set is not necessarily whole manifold. This holds because the minimization problem, as  \eqref{minoptprob1}, can be formulated as an equilibrium problem where the bifunction, associated to this particular problem, satisfies \eqref{ineq5}  with $\rho=1\in  ]0,1]$, or equivalently,  $S=argmin _{\Omega}\phi$ is a weak sharp minima set for the minimization problem \eqref{minoptprob1} (see \eqref{ineq:weaksharp11-2013}). For a characterization of the condition \eqref{ineq:weaksharp11-2013}, in the case of convex minimization problems on Hadamard manifolds,  see \cite{LMWY2011}.
\item [ii)]  In regards to other problems that may be interpreted as equilibrium problems, we emphasize that our resulted of finite convergence can be applied  for the proximal point methods to finding singularities of single valued monotone vector fields (see \cite{CFPN2006}), multivalued monotone vector fields (see \cite{LLMM2009}) and, hence, for the  variational inequality problem (see \cite{N2003}) among others, when   the bifunction associated to each of these problems satisfies \eqref{ineq5}  with $\rho\in ]0,1]$.  We intend, in future work, to identify the condition \eqref{ineq5} in each particular instance listed earlier, and to investigate possible characterizations. 
\end{itemize}
\end{remark}


In order to illustrate the result of Theorem \ref{teo4}, let us consider the elementary Example \ref{ex1}. In this case, $S=EP(F,\Omega)=\{1\}$ and     
\[
F(x^{k+1},1)=x^{k+1}(1-x^{k+1})=x^{k+1}dist(x^{k+1},S)\geq \frac{1}{2}dist(x^{k+1},S),
\]
where $\rho=1$ and $\tau=1/2$. Therefore, the sequence $\{x^k\}$ converges to $x^\ast=1$ in a finite number of \mbox{iterations}.

\section{Worthwhile Transitions to Variational Traps on an Hadamard Manifold}

In this section, devoted to applications, we consider stability and
change dynamics. They are everywhere. For example, among many others, habits
and routines formation and break, and exploration-exploitation
stability/change issues (where exploration refers to discovery, search,
innovation\ldots).  A recent ``variational rationality" approach in Behavioral
Sciences (see \cite{S2009,S2010}) modelized and unified a huge list of such
stability and change dynamics. It shows how a succession of worthwhile stays
and changes can converge to a variational trap (to be defined below). In an
important paper, Colao et al.~\cite{CLMM2012} gives, first in a static framework,
and then using an approximation process, three applications of equilibrium
problems (EP) on Hadamard manifolds: mixed variational inequalities (MVI),
fixed points (FP) of set valued mappings, and Nash equilibrium (NE) in non
cooperative games. In this last section related to applications, we show
that their approach represents a special, but very important benchmark case,
of a worthwhile stability and change dynamic on an Hadamard manifold. The
Variational rationality (VR) approach (see \cite{S2009,S2010}) rests on three
main concepts: worthwhile single stays and changes, worthwhile transitions
(defined as successions of worthwhile single stays and changes), and
stationary/variational traps. \ The plan of this section works as follows: We
\begin{itemize}
\item [1)] start our application with a critic of a Nash equilbirium as a purely
static concept, which have no real existence in a dynamic setting. In
contrast, we define, in an informal way, the concept of a variational trap,
which is both, i) a stationary trap (a stable position with respect to any
worthwhile change), and ii) a reachable position, starting from the initial
position, following a worthwhile transition. Then, to show the importance of
this application, we give, among many other examples, a short list of \
traps in Behavioral Sciences (Psychology, Economics, Management Sciences,
Sociology, Political Sciences, Game theory, Decision theory, Artificial
Intelligence,\ldots) and Applied Mathematics;

\item [2)] show why the context of an Hadamard manifold is so important for the
examination of stability and change dynamics;

\item [3)] summarize the Variational rationality (VR) approach (\cite{S2009,S2010}) in the context of this paper, to be able to define and modelize the
three main (VR) concepts of worthwhile single stays and changes, worthwhile
transitions and stationary/variational traps in this context;

\item [4)] examine, in the Euclidian case, the variational trap problem, which
refers to the possible convergence of a succession of worthwhile single
changes and stays, moving from a stationary trap to the next, ending in a
variational trap;

\item [5)] examine the variational trap problem in the case of an Hadamard manifold;

\item [6)] show the Variational rationality flavor of all the hypothesis done in
this paper;

\item [7)] Finally, we interpret exact and inexact solutions to the equilibrium
problem in variational rationality terms.
\end{itemize}

\subsection{Critic of the Nash Equilibrium Static Concept in a Dynamic Setting}

The Variational rationality (VR) approach (see \cite{S2009,S2010}) makes a
critic of an optimum, and more generally of a Nash equilibrium, as a purely
static concept, a stability concept, where agents, being at equilibrium,
prefer to stay there than to move, while they do not know why they are
there, by pure change, an event of zero measure\ldots. On the contrary, the
(VR) approach urges to focus attention, for dynamical applications in
Behavioral Sciences, on traps, and more precisely variational traps. These
variational traps satisfy two conditions:
\begin{itemize}
\item [i)] stability issue: they are stable with respect to worthwhile changes
(stationary traps): being there, at equilibrium, agents prefer to stay than
to move because to stay is worthwhile and to move is not. Nash equilibrium are
also stable positions, not with respect to worthwhile changes, but with
respect to advantages to change (ignoring resistance to change, in a world
without frictions). It is as if agents, at the very beginning of the process,
by luck, find them in the trap;
\item [ii)] reachability and desirability issues with respect to a given subset of
initial positions: traps must be also reachable and desirable to reach,
following a succession of worthwhile single stays and changes (a worthwhile
transition), moving from the initial position to the trap. In this general
case, the agent is not in the trap at the beginning of the process, as a
static equilibrium supposes. Then, the agent prefers to move towards the
trap than to stay, following a succession of worthwhile single stays and
changes.
\end{itemize}
This shows that a Nash equilibrium is a very specific case of a variational
trap, which is difficult to justify in dynamic settings. A position can be
stable, like a Nash equilibrium, but will have no practical sense if it is
not reachable from the initial position. In this case this is an empty
concept.

\subsection{Examples of Traps in Behavioral Sciences}

In Behavioral Sciences, Alber and Heward in~\cite{AH1996} noticed that \textquotedblleft relatively simple response is necessary to enter the trap, yet once entered, the trap
cannot be resisted in creating general behavior changes\textquotedblright. Among a huge list of traps, in a lot of different disciplines, let us give some of them:

\begin{itemize}
\item[A)] In Psychology, Baer and Wolf in~\cite{BW1970} seem to be the first to use the term behavioral trap in term of the end of a reinforcement process, describing
\textquotedblleft how natural contingencies of reinforcement operate to promote and maintain
generalized behavior changes\textquotedblright. Much later, Plous in~\cite{P1993} defines traps as
more or less easy to fall into and more or less difficult to get out. He
lists five main behavioral traps: investment, deterioration, ignorance and
collective traps\ldots. Then, several authors emphasized, again, that
behavioral traps are the ends of reinforcement processes, see Stephen \cite{S2004}. Also, Baumeister in~\cite{B2002}, and Baumeister and Heatherton in~\cite{BH1996} consider ego depletion traps, due to fatigue costs, in the context of self regulation
failures. Cognitive and emotional traps refer to all-or-nothing thinking,
labeling, over generalization, mental filtering, discounting the positive,
jumping to conclusions, magnification, emotional reasoning, should and
shouldn't statements, and personalizing the blame \ldots.

\item[B)] In Economics and Decision Sciences, traps refer to hidden biases and
heuristics. For example anchoring, status quo, sunk costs, confirming
evidence, framing, estimation and forecasting traps, see~Hammond et al. \cite{HRKR1998}. Traps also represent habits and routines.

\item[C)] In Management Sciences, Levinthal and March in~\cite{LM1993} define, at the
organizational level, success and failure traps, in the context of the so
called \textquotedblleft  myopia of learning\textquotedblright effect.

\item[D)] In Development theory, poverty traps is a main issue. To explain their
formation, Appadurai in~\cite{A2004} defined aspiration traps which describe the
inability to aspire of the poors, (see Heifetz and Minelli~\cite{HM2006}, Ray~\cite{R2006}).
\end{itemize}

Traps refer to rather easy to reach (feasibility and desirability issues)
and difficult to leave (stability issue) positions. They generalize critical
points, optima, Nash equilibria, fixed points, Pareto optima, and mixed
variational equilibria, as reachable, desirable and stable end points of a
dynamical process. They can represent habits, routines, rules, conventions,\ldots.

\subsection{Hadamard Manifolds. The Modelization of the Repeated Regeneration
of Resources.}

A striking advantage to consider equilibrium problems on an Hadamard
manifold is to allow to consider, each period, repeated constraints, which
balance the repeated depletion and regeneration of resources, helping to
modelize human behavior in a realistic dynamic setting. An important
constraint, almost always neglected in the economic literature, is that the
agent must spend, and then, regenerate his depleted energy as time evolves
(see the Ego-depletion theory in Psychology, Baumeister in~\cite{B2002}). To
satisfy their needs, agents must spend effort and energy to gather means in
order to be able to acquire (produce, buy, sell, eliminate, \ldots ) final
goods which will satisfy partially these needs. Then, each period, the lost
energy must be recovered.

Consider an agent who performs, each period, a list of activities $%
x=(x^{i},x^{j})$, where activities $i\in I$ produce daily vital energy for
the agent (like eating, resting, holidays, sports, healthy activities, arts,
\ldots ), giving further motivations to act, and activities $j\in J$ consume
energy (working, thinking, \ldots ). Let $e_{+}^{i}(x^{i})\geq 0$ be the
energy produced by doing action $x^{i}\in \mathbb{R}$ and $%
e_{-}^{j}(x^{j})\geq 0$ be the energy consumed by doing action $x^{j}\in 
\mathbb{R}$. Then, the regeneration of vital energy imposes the constraint $%
\Sigma _{i\in I}e_{+}^{i}(x^{i})-\Sigma _{j\in J}e_{-}^{j}(x^{j})=E>0$.
Production and consumption functions of energy can be increasing and convex
(the more an agent carries out an activity, the more he produces and
consumes energy, at an increasing rate). In the quadratic case, the
expression $\Sigma _{i\in I}(x^{i})^{2}-\Sigma _{j\in J}(x^{j})^{2}=E>0$
defines an hyperboloid. A more realistic example can be given where
production functions of energy are increasing, concave, and consumption
functions of energy are increasing convex. More generally each activity can
both consume and produce some energy.

\subsection{The VR Approach. A Model for Worthwhile Changes and Variatonal
Traps}

It is time to define and modelize, right now, the three main (VR) concepts,
worthwhile single changes and stays, worthwhile transitions, and variational
traps. The (VR) variational rationality approach ~\cite{S2009,S2010}
advocates that agents are \textquotedblleft variationally rational".
Following the famous Prospect theory with riskless choice in Economics
(Tversky and Kahneman~\cite{TK1991}) this \textquotedblleft variational
rationality" approach considers that agents look at changes, rather than at
endowments, stocks and wealth to evaluate their current utility. Then, it
goes a step further, in a true dynamic context, where the past, the present
and the future matter to determine a behavior (a succession of actions). It
emphasizes, first, that the main variational question is \textquotedblleft
should I stay, should I go?". Then, it advocates that, a lot of times,
agents do not optimize (contrary to Kahneman and Tversky~\cite{KT1979}, who
suppose that they do). Instead it considers that agents can
\textquotedblleft muddle through" (see Lindboom~\cite{L1959} in Political
Sciences), behave as bounded and procedural rational agents (Simon~\cite%
{S1955}, in Economics and Management), as well as \textquotedblleft
practical rational" agents (Bratman~\cite{B1981} in Philosophy, Wooldridge~%
\cite{W2000}, in Artificial Intelligence). To unify all these different
points of view in a lot of different disciplines, the variational
rationality approach considers that, each step of a behavioral process
(defined as a succession of actions), an agent tries to perform worthwhile
changes or stays.

Let us summarize very succintly the main aspects of this recent
\textquotedblleft variational rationality" approach. For simplification, let
us consider the case of an agent. For more complex situations and a lot of
variants, see~\cite{S2009,S2010}. The case of a game with interrelated
agents will be examined below, to compare our findings with those of \cite{CLMM2012}. Given an agent and his current experience $e\in E$ (which
depends of the sequence of his past actions including the last action $z$
which has been done$,$ his motivation to change $\mathcal{M}_{e}(z,y)\in 
\mathbb{R}$ from repeating the old action $z,$ to do a new action $y,$ must
be higher than a choosen and satisficing worthwhile to change ratio $\lambda
>0$ time his resistance to change $R_{e}(z,y)\in \mathbb{R}_{+}$. This ratio
is adaptive. Then, the core of this construction is the following
\textquotedblleft worthwhile to change or stay" inequality: if $z$ is the
last past action which have been done and $y$ is a new action planned to be
done in a near future, it is worthwhile to change ($z\curvearrowright y)$
than to stay ($z\curvearrowright z)$ iff $\mathcal{M}_{e}(z,y)\geq \lambda
R_{e}(z,y)$ where:

\begin{itemize}
\item[-] $\mathcal{M}_{e}(z,y)=U\left[ A_{e}(z,y)\right] $ is the utility or
pleasure, $U\left[A_{e}\right]$, of advantages to change rather than to
stay, $A_{e}=A_{e}(z,y)$;

\item[-] $R_{e}(z,y)=D\left[ I_{e}(z,y)\right]$ is the desutility or pain, $D%
\left[I_{e}\right] ,$ of the inconvenients to change $I_{e}=I_{e}(z,y)\geq
0. $
\end{itemize}

The worthwhile to change rather than to stay payoff, at $z,$ is $\Delta_{\lambda,e}(z,y)=\mathcal{M}_{e}(z,y)-\lambda R_{e}(z,y)$ and the worthwhile
\textquotedblleft change rather than to stay" set, at $z$ is $W_{\lambda
,e}(z)=\left\{ y\in \Omega ,\Delta _{\lambda ,e}(z,y)\geq 0\right\} $.

The goal of this variational approach is to examine the dynamics of a
succession of worthwhile stays and changes $x^{k+1}\in W_{\lambda
_{k+1,e^{k}}}(x^{k})$, $\lambda _{k+1}>0$, $k\in \mathbb{N}$, where $%
e^{k}=E(X^{k})\in E$ is the experience of the agent at step $k$, which
depends of the history of past actions $X^{k}=(x^{0},x^{1},\ldots ,x^{k})$,
as well as, in more complex cases, the history of the actions of other
agents and the environment. Each step, given the current action $x^{k}$, the
agent chooses first a satisficing worthwhile to change ratio $\lambda
_{k+1}>0$ (in order to consider a change as worthwhile this step) and, then,
tries to find a new worthwhile to change action $x^{k+1}\in W_{\lambda
_{k+1,e^{k}}}(x^{k})$, which must belong to the worthwhile to change set.

The main questions are, depending on the evolving context (the given
parameters, each step):

\begin{itemize}
\item[i)] does this process converges, where does it converges, what is the
end (a variational trap where it is worthwhile to stay than to move);

\item[ii)] what are the dynamical properties of the transition from the
given initial point to the end: speed of convergence, convergence in finite
time, \ldots ;

\item[iii)] how efficient is this process: does the end is a critical point,
a local maximum or minimum, an approximate equilibrium, an equilibrium
\ldots?
\end{itemize}

\textbf{Payoff functions.} Let $g_{e}:M\rightarrow \mathbb{R}$ be a payoff
function (performance, revenue, profit \ldots ) to be \mbox{maximized} and $%
f_{e}:M\rightarrow \mathbb{R}$ be a cost function, or an unsatisfied need
function to be minimized, which, both, will depend of the experience $e\in E$
of the agent.

\textbf{Advantages to changes}. In the separable case, advantages to change
are $A_{e}(z,y)=g_{e}(y)-g_{e}(z)$ or $A_{e}(z,y)=f_{e}(z)-f_{e}(y).$ More
generally the advantage to change function is a bifunction $A_{e}:M\times
M\rightarrow \mathbb{R}$.

\textbf{Costs to be able to change and costs to be able to stay. }Given the
experience $e\in E$\textbf{, }costs to be able to change (moving from being
able to repeat the last action $z$ to be able to do a new action $y$) are $%
C_{e}(z,y)$. Costs to be able to stay at $z$ ( i.e to repeat action $z$) are 
$C_{e}(z,z)\geq 0.$

\textbf{Inconvenients to change. }Let $%
I_{e}=I_{e}(z,y)=C_{e}(z,y)-C_{e}(z,z)\geq 0$ be the inconvenients to change
between repeating the same old action $z$, and doing a new action $y$. They
represent, given the agent's experience $e\in {E}$, the difference between
costs to be able to change $C_{e}(z,y)$ and costs to be able to stay, $%
C_{e}(z,z)\geq 0.$

\textbf{Stationary traps (strong and weak)}. Let $\lambda _{\ast }>0$ and $%
e^{\ast }\in E$ be a given satisficing worthwhile to change ratio and a
given experience. Then, $x^{\ast }\in \Omega $ is a strong stationary trap
if 
\[
\Delta _{\lambda _{\ast },e^{\ast }}(x^{\ast },y)=\mathcal{M}_{e^{\ast
}}(x^{\ast },y)-\lambda _{\ast }R_{e^{\ast }}(x^{\ast },y)<0,\qquad y\neq
x^{\ast }. 
\]

This means that it is not worthwhile to move from $x^{\ast },$ i.e, the
worthwhile to change set shrinks to a point, $W_{\lambda _{\ast },e^{\ast
}}(x^{\ast })=\left\{ x^{\ast }\right\}$. This refers to a stability issue,
which takes also care of some resistance to change at $x^{\ast }$. Notice
that a (strong) equilibrium $x^{\ast }\in \Omega $ will not consider any
resistance to change, and will refer to the truncated condition not
motivation to change condition $\mathcal{M}_{e^{\ast }}(x^{\ast },y)<0$, for all $y\neq x^{\ast }$. A weak stationary trap is such that $$\Delta_{\lambda
_{\ast },e^{\ast }}(x^{\ast },y)=\mathcal{M}_{e^{\ast }}(x^{\ast},y)-\lambda _{\ast }R_{e^{\ast }}(x^{\ast },y)\leq 0,\qquad y\in \Omega.$$

\textbf{Variational traps. }An action $x^{\ast }\in X$ is a strong (weak)
variational trap with respect to an initial action $x_{0}$ if, starting from 
$x_{0}$, it exists a succession of \textquotedblleft changes and stays" $%
x^{k+1}\in W_{\lambda _{k+1},e^{k}}(x^{k})$, $k\in \mathbb{N}$, defined by a
convergent sequence of satisficing worthwhile to change ratio and
experiences, 
\[
(\lambda _{k+1},e^{k})\rightarrow (\lambda _{\ast },e^{\ast }),\
k\rightarrow +\infty , 
\]%
such that,

\begin{itemize}
\item[i)] $x^{k}$ converges to this point $x^{\ast }$ (it is worthwhile to
move to this trap $x^{\ast }$, starting from $x^{0}$);

\item[ii)] $x^{\ast }\in \Omega $ is a strong stationary trap, $W_{\lambda
_{\ast },e^{\ast }}(x^{\ast })=\left\{ x^{\ast }\right\} $ (weak stationary
trap).
\end{itemize}

In general a variational trap will be relative to a source set, an initial
subset of actions $\Omega _{0}\subset M$, instead of a given initial action $%
\Omega _{0}=\left\{ x_{0}\right\} .$ This means that, the trap is worthwhile
to reach, starting from any point of the source set $\Omega_{0}$. Notice
that the second condition ii) defines only a strong stationnary trap (a
stability issue, when inertia matters). The first condition tells us that it
is worthwhile to move to this trap $x^{\ast }$, starting from $x^{0}$ (a
feasible and acceptable reachability issue).

\subsection{The Variational Trap Problem in the Euclidian Space}

The variational trap problem refers to the possible convergence of a
succession of worthwhile single stays and changes, moving from a stationary
trap to the next, ending in a variational trap.
\begin{itemize}
\item \textbf{Equilibrium problems and Variational rationality equilibrium
problems.} 
Consider the case where the space of actions is the
Euclidian manifold $M=\Omega =\mathbb{R}^{n}$. 
The equilibrium problem (EP) given in~\eqref{ep}) refers, in Mathematics, to costs and losses
minimization.
Given a differentiable bifunction $A:\Omega \times \Omega \rightarrow \mathbb{R}$, the variational rationality equilibrium problem (VR-EP) refers,
in Behavioral Sciences, to gains and advantages to change maximization (see~\cite{S2009,S2010}). It is: find $x^{\ast }\in \Omega $ such that the
possible advantage to change $A(x^{\ast },y)$ be non positive, i.e, $%
A(x^{\ast },y)\leq 0$, for all $y\in \Omega .$ When $F=-A,$ the (VR-EP)
problem is equivalent to the (EP) equilibrium problem. For easier
comparisons with \cite{CLMM2012}, start from the Mathematical
equilibrium problem (EP). In this last section devoted to applications,
losses and advantages to change do not depend of experience, i.e, $%
F_{e}(x,y)=F(x,y)$ and $A_{e}(x,y)=A(x,y)$ for all $e\in E.$ The simplest
separable and experience independent case, where
\[F(x^{\ast
},y)=f(y)-f(x^{\ast })=g(x^{\ast })-g(y)\geq 0, \, \mbox{for all} \; y\in \Omega,
\]
defines a minimum $x^{\ast }\in \Omega $ of the unsatisfied need function $f$, or a maximum of the payoff function $g$ over $\Omega .$

\item \textbf{Changes. }Let $\Omega $ be the space of actions and consider an
agent who, each period, looks at three consecutive actions, $x,z,$ $y\in
\Omega $, where $x$ and $z$ represent past actions, the two last actions he
has done, and $y$ is the future action he plans to do. Hence, in this
specific context, before looking at doing the future action $y$, the current
experience of the agent is $e=(x,z)\in \Omega \times \Omega ={E}$. If the
agent plans to do an action $y$ which is less similar to the past action
than the present action, we will suppose that his costs to be able to do
this new action $y$ increases with respect to the cost he has spend to be
able to do the last past action $z$. In the opposite case this cost will be
higher.

\item \textbf{Pleasure and Pain functions. }In the specific context of this paper,
the equilibrium problem (EP) on an Hadamard manifold, identify pleasure
(utility) to advantages to change, ${U}\left[ A_{e}\right] =A_{e}$, and
pain (desutility) to inconvenient to change, ${D}\left[ I_{e}\right] =I_{e}$.

\item \textbf{Advantages to changes. }They are modelized as the bifunction $%
A:M\times M\rightarrow R$. It is a constant function which does not depend
of experience $e\in {E}$.

\item \textbf{Losses to change. }They refer to the bifunction $F:M\times
M\rightarrow R$, where $F(z,y)=-A(z,y)$\textbf{.}

\item \textbf{Costs to be able to change and costs to be able to stay. }They are $%
C_{e}(z,y)\geq 0$ and $C_{e}(z,z)\geq 0$. Notice that costs to be able to
stay are not necessary zero.

\item \textbf{Inconvenients to change as Tikhonov regularization terms. }In the
context of this paper, we modelize inconvenients to change $%
I_{e}(z,y)=C_{e}(z,y)-C_{e}(z,z)\lessgtr 0$ as a Tikhonov perturbation term,
given by $I_{e}(z,y)=$ $\langle z-x,y-z\rangle $, where the experience $e=(x,z)$ of the agent concerns only his last two past actions, $x$ and $z$
(both have been done) 
It turns out that, in Behavioral Sciences, it is identical to the
\textquotedblleft Cosinus similarity" index, which is a measure of
similarity between the two vectors $u=z-x,$ and $y-z$ , where $v=z-x$
represents the present change in past experience, and $y-z$ the future
change. As a scalar product, it is equal, up to the sign (+, or -), to the
distance $d(x,z)$ between $x$ and $z$, time the distance $d(z,y_{p})$
between $z$ and the projection $y_{p}$ of $y$ on the lign $L(u)$ supported
by the vector $u=z-x,$ i.e, $I_{e}(z,y)=d(x,z)\times d(z,y_{p})\geq 0$ if $%
y_{p}$ is outside the segment $\left[ z^{\prime },x,z^{\prime \prime }\right]
$ on the lign $L(d)$, i.e., 
\[
y_{p}<-----\ z=z^{\prime }<------\ \ \ x----->z=z^{\prime \prime
}------>y_{p} 
\]%
and, $I_{e}(z,y)=-d(x,z)\times d(z,y_{p})\leq 0$ if $y_{p}$ is inside the segment 
$\left[ z^{\prime },x,z^{\prime \prime }\right] \ $ on the lign $L(d)$,
i.e., 
\[
z=z^{\prime }<-----\ y_{p}<------\ \ x----->y_{p}------>z=z^{\prime \prime } 
\]%
The interpretation is clear: having in mind that, $x=$ two periods less
recent past action, $z=$ one period most recent past action, and $y=\ $one
period future action, the costs to be able to change rather than to stay
increase (decrease) if the dissimilarity between the future action $y$ and
the second past action $x$ is higher (lower) than the dissimilarity between
the last past action $z$ and the past action $x$. This index shows in a
striking way how past experience matters much to determine future costs to
be able to change.

\item \textbf{Worthwhile changes}. Let $\Delta _{\lambda ,e}(z,y)=A(z,y)-\xi
I_{e}(z,y)=$ $-\left[ F(z,y)+\lambda I_{e}(z,y)\right] $ be the
\textquotedblleft worthwhile to change rather than to stay" payoff of the
agent. In  this paper, when  $F(z,y)=-A(z,y)$, $M=\mathbb{R}^{n}$ and $I_{e}(z,y)=$ $\langle z-x,y-z\rangle ,$ $\lambda >0$
is a satisficing worthwhile to change ratio, and $e=(x,z)\in {E}$ modelizes
experience (the sequence of two actions he has done before, including the
current action, done just before). The present paper considers the specific
case of short memory, where the agent \textquotedblleft remembers" only his
two last past actions $x,$ $z$. A change from the last past action $z$ (yet
done) to the new action $y$ (to be done), is worthwhile if $\Delta _{\lambda
,e}(z,y)\geq 0$. This means that advantages to change $A(z,y)$ from $z$ to $y$
are higher than $\lambda >0$ times the experience dependent inconvenients to
change $I_{e}(z,y)$, i.e, $A(z,y)\geq \lambda I_{e}(z,y)$ which is
equivalent to a non negative sum of losses to change plus the inconvenients
to change, $\ F(z,y)+\lambda I_{e}(z,y)\geq 0.$

\item \textbf{Successions of worthwhile changes. }At period $k,$ the two last past
actions and future actions are $x=x^{k-1},z=x^{k}$ and $y=x^{k+1}.$ Then, a
change from $z=x^{k}$ to $y=x^{k+1}$ is worthwhile if:
\[
\Delta _{\lambda
_{k+1,}e^{k}}(x^{k},x^{k+1})=A(x^{k},x^{k+1})-\lambda
_{k+1}I_{e^{k}}(x^{k},x^{k+1})\geq 0.
\]  
A succession $\left\{ x^{k}\right\} $of worthwhile single stays and changes
satisfies $\Delta _{\lambda _{k+1,}e^{k}}(x^{k},x^{k+1})\geq 0$ for all $%
k\in \mathbb{N}$. To fit with the mathematics, we can note $\lambda _{k}>0$ instead
of $\lambda _{k+1}>0.$ In the second mathematical case $\lambda _{k}>0$ is
given, and refers to an heritage from the past and current situations; In
the first behavioral case, the satisficing worthwhile to change ratio $%
\lambda _{k+1}>0$ is chosen, and refers to a future change from $z=x^{k}$
and $y=x^{k+1}.$

\item \textbf{Stationary trap (strong and weak).} A strong (resp. weak) stationary trap $x^{\ast }\in \Omega $\textbf{\ }
is such that, starting from $x^{\ast }$, there is no available worthwhile
change to be done:
\[
\Delta _{\lambda _{\ast },e^{\ast }}(x^{\ast },y)=A(x^{\ast },y)-\lambda
_{\ast }I_{e^{\ast }}(x^{\ast },y)<0,\qquad y\neq x^{\ast},\quad
(resp.\; \Delta_{\lambda _{\ast },e^{\ast }}(x^{\ast },y)\leq 0,\quad y\in \Omega).
\]

\item \textbf{Resolvents \ as subsets of weak stationary traps.}
Let us consider the following set:
\[
J_{\lambda }^{A}(x):=\left\{ z\in \Omega: 
A(z,y)\leq \lambda I_{e}(z,y), \; \forall y\in \Omega\right\},
\]
which is a variational rationality representation of resolvent of $F$ defined in \eqref{eq:resolventColao}.
They represent the set of weak stationary traps related
to experience $e=(x,z):=(z-x)\in E$.

\begin{lemma}[Resolvent Lemma]. If $z\in J_{\lambda }^{F}(x)$ and if advantages to changes are monotone, i.e, ($A(y,z)+A(z,y)\geq 0$ for
all $y,z\in \Omega $), then, it is worthwhile to change from $x$ to $z,$ i.e, 
$z\in W_{\lambda ,e}(x)$ where $e$ refers to $z-x$.
\end{lemma} 
\begin{proof} Take $y=x$. From $e=z-x$ and  $I_{e}(z,y)=\langle e,y-z\rangle$, we get
$$
I_{e}(z,x)=\langle z-x,x-z\rangle =-\left\Vert z-x\right\Vert
^{2}=-I_{e}(x,z).
$$
Hence, taking into account that $z\in J_{\lambda }^{F}(x)$, $\langle \exp^{-1}_{z}{x},\exp^{-1}_{z}x\rangle=\langle z-x, z-x\rangle$, $A=-F$ and given the monotonicity of $A$, we have
\[
A(x,z)\geq -A(z,x)=F(z,x)\geq \lambda I_{e}(z,x)=-\lambda I_{e}(x,z),
\]
and the affirmation is proved.
\end{proof}

\item \textbf{A strong (weak) variational trap }$x^{\ast}\in \Omega$ is such that, given a worthwhile to change process $\{x^{k}\}$ converging to $x^{\ast}$, we have:
\begin{itemize}
\item [i)] $\Delta_{\lambda_{k+1,}e^{k}}(x^{k},x^{k+1})=A(x^{k},x^{k+1})-\lambda_{k+1}I_{e^{k}}(x^{k},x^{k+1})\geq 0, \qquad k\in \mathbb{N}$;  
\item [ii)] $\Delta_{\lambda_{\ast }, e^{\ast }}(x^{\ast},y)=A(x^{\ast },y)-\lambda_{\ast }I_{e^{\ast }}(x^{\ast },y)<0,\qquad y\neq x^{\ast }.$
\end{itemize}
\end{itemize}
\subsection{The Case of an Hadamard Manifold}
\begin{itemize}
\item \textbf{The consideration of regeneration of resources constaints}. As said
before, the context of Hadamard manifolds is fundamental for the
consideration of dynamic problems. It allows to consider regeneration of
resource constraints, which are almost always neglected in dynamic models,
where the state space is an Euclidian space.

\item \textbf{The modelization of inconvenients to change.} The main
\textquotedblleft variational rational" concept whose formulation changes when the space of actions moves from $M=\mathbb{R}^{n}$ to an
Hadamard manifold is the inconvenient to change function 
$I_{e}(z,y)=C_{e}(z,y)-C_{e}(z,z)$ which passes from 
\[
I_{e}(z,y)=C_{e}(z,y)-C_{e}(z,z)=\langle z-x,y-z\rangle \;\; \mbox{to}\;\; I_{e}(z,y)=\langle \exp _{z}^{-1}x,\exp _{z}^{-1}y\rangle.
\]
The interpretation is exactly the same as before and needs not to be repeated
because (see Proposition~\ref{T:lawcos}) the comparison theorem for
triangles establishes a diffeomorphism between a geodesic triangle $\Delta
(x,y,z)$ on the Hadamard manifold and a corresponding one, $\Delta
(x^{\prime },y^{\prime },z^{\prime })$ in the Euclidian space $E=\mathbb{R}%
^{n}$. Then, $\langle \exp _{z}^{-1}x,\exp _{z}^{-1}y\rangle $ $%
=d(x,z)d(y,z)\cos \alpha .$ As before costs to be able to change depend of
geodesic distances.

\item \textbf{The variational trap problem.}

We only need to verify that, moving from a weak stationary trap to a new
one, is a worthwhile change. Then, the Euclidian versions of Theorems~\ref{teo3} and \ref{teo4} show that limits of a succession of worthwhile changes $\left\{
x^{k}\right\} $ from a weak stationary trap $x^{k}\in J_{\lambda
_{k}}^{A_{k-1}}(x^{k-2})$ to the next $x^{k+1}\in J_{\lambda
_{k+1}}^{A_{k}}(x^{k-1})$ converges towards an equilibrium $x^{\ast }\in
\Omega $.

The conditionality assumption (18) is a weak sharp minimum condition. It
supposes that $x^{\ast }$ is a weak variational trap, related to the new
resistance to change function $R(z,y)=\mbox{dist}^{\rho}(x,S)$; see the next
paragraph.

The resolvant Lemma shows that it is worthwhile to move from a weak
variational trap to the next.

\item \textbf{Finite termination} is a very important property for a model which
wants to modelize human behavior in a nice way, because, in the long run, we
are dead! Convergence in infinity time is of no use to describe a goal
directed behavior who requires to hope to reach a goal in finite time.
\end{itemize}

\subsection{The Variational Rationality Flavor of all the Hypothesis Done in
this Paper.}

Let us show that all the hypothesis made in this paper have a strong
variational rationality flavor. 
The variational rationality (VR) approach considers, among other basic concepts, advantages
to change $A(x,y)$ from $x$ to $y$. Losses to change from $x$ to $y,$ $%
F(x,y),$ refer to the opposite, the losses to change from $x$ to $y$ which
is $F(x,y)=-A(x,y).$ Within the (VR) approach, the equilibrium problem in a
Riemannian context is: 
\[
\mbox{find} \quad x^{\ast }\in \Omega:\quad A(x^{\ast
},y)\leq 0, \qquad y\in \Omega.
\]
Then, an equilibrium $x^{\ast }\in
\Omega $ is such that there is no advantages to change (there is no gain to
deviate from $x^{\ast }$), or there are losses to change, moving from it to
an other position. Definition 3.1 considers a loss function $F:\Omega \times
\Omega \longmapsto \mathbb{R}$. This loss function $F$ is said to be,
\begin{itemize}
\item [(1)] monotone, iff $F(x,y)+F(y,x)\leq 0$ for all $x,y\in \Omega $, i.e $%
A(x,y)+A(y,x)\geq 0$ for all $x,y\in \Omega$. This means that, moving from $%
x $ to $y$ and coming back, cannot generate a global loss to change along
this cycle (in term of advantage to change). Notice that costs to be able to
change are excluded from the very definition of advantages to change;

\item [(2)] pseudomonotone, iff for each $(x,y)\in \Omega \times \Omega ,F(x,y)\geq
0 $ \ implies $F(y,x)\leq 0$, i.e., $A(x,y)\leq 0$ \ implies $A(y,x)\geq 0.$%
This means that it is always advantageous to move from $x$ to $y$ or the
reverse. This is a no strong indecision hypothesis;

\item [(3)] $\theta $-undermonotone, iff there exists $\theta \geq 0$ such that \[
F(x,y)+F(y,x)\leq \theta d(x,y)^{2}
\], i.e., $A(x,y)+A(y,x)\geq -\theta
d(x,y)^{2},$ for all $(x,y)\in \Omega \times \Omega$. This means that the
global advantage to change from $x$ to $y$ and to come back cannot be too
low, the more actions $x,y$ are similar, the less it is. This is a low
resistance to change hypothesis.
\end{itemize}
These three properties are trivially true in the separable and experience
independent advantages to change case, $A(x,y)=g(y)-g(x)=f(x)-f(y)$, i.e $%
F(x,y)=g(x)-g(y)=f(y)-f(x),$ where $g$ and $f$ are the ``to be increased" or
"to be decreased" payoff functions, for a single agent. In the general case $%
\ F=-A$ is pseudomonotone if, having an advantage to change from $x$ to $y,$
there is no advantage to change from $y$ to $x$. Consider now the case of
interacting agents $I$ playing a Nash non cooperative game (see \cite{CLMM2012}), where the profile of their actions is $%
x=(x_{1},x_{2},..,x_{i-1},x_{i},x_{i+1},....x_{m}).$ In this setting, player 
$i\in I$ performs an action $x_{i}\in K_{i},$ while his rivals $-i=(j\neq i)$
carry out the other actions $x_{-i}=(x_{j},j\neq i),$ and each player $i$
considers his unsatisfied need function $%
f_{i}(x_{1},x_{2},..,x_{i-1},x_{i},x_{i+1},....x_{m})=f_{i}(x).$

Let $F(x,y)=\Sigma _{i\in I}\left[
f_{i}(x_{1},x_{2},..,x_{i-1},y_{i},x_{i+1},....,x_{m})-f_{i}(x_{1},x_{2},..,x_{i-1},x_{i},x_{i+1},....x_{m})%
\right] $ be the perceived global loss function of the players, where each
player is supposed to be able to only perceive what will be his loss $%
L_{i}((x_{i},x_{-i}),(y_{i},x_{-i}))=\left[
f_{i}(x_{1},x_{2},..,x_{i-1},y_{i},x_{i+1},....,x_{m})-f_{i}(x_{1},x_{2},..,x_{i-1},x_{i},x_{i+1},....x_{m})%
\right] ,$ if he is the only agent to move from the profile of old actions $%
(x_{1},x_{2},..,x_{i-1},x_{i},x_{i+1},....x_{m})=(x_{i},x_{-i})$ to the new
profile of actions $%
(x_{1},x_{2},..,x_{i-1},y_{i},x_{i+1},....x_{m})=(y_{i},x_{-i})$, instead of
choosing to stay there, at $x=(x_{i},x_{-i}),$ given that he supposes that
the other players repeat their old actions.

In regard to the assumptions $\mathcal{H}$1, $\mathcal{H}$2, $\mathcal{H}$3 on the loss functions, we have
\begin{itemize}
\item[$\mathcal{H}$1)] $F(x,x)=0$ for each $x\in \Omega $ means \ that there
is no loss to stay;

\item[$\mathcal{H}$2)] For every $x\in \Omega $, $y\mapsto F(x,y)$ is convex
and lower semicontinuous, i.e., the advantages to change function $%
y\longmapsto A(x,y)$ is concave and upper semicontinuous;

\item[$\mathcal{H}$3)] For every $y\in \Omega $, $x\mapsto F(x,y)$ is upper
semicontinuous, i.e., the advantages to change $x\mapsto A(x,y)$ is lower semicontinuous.
\end{itemize}
Lower (upper) and upper (lower) semicontinuity assumptions are natural
regularity assumptions for loss (gain) functions. Only convexity (concavity)
of the loss (gain) function $y\mapsto F(x,y)$ needs some comments. Given the
past action $x$, $A(x,y)=-F(x,y)$ is the gain to move from $x$ to $y.$
Concavity of $A(x,\cdot)$ refers to the traditional assumption of decreasing
marginal gains (a standard saturation effect).

Consider now the remaining hypothesis. It is not easy to give a variational
interpretation of Assumption~\ref{assumption-13-1}. However, let us notice that when $F$ is
pseudomonotone, Assumption \ref{assumption-13-1} is satisfied. Then, for an interpretation,
we will refer to pseudomonotonicity. Now,  Assumption~\ref{assumption-13-2} means that any (unbounded) sequence $\left\{ z^{k}\right\} \subset\Omega$ whose distance to an initial position goes to infinity, have, in
term of the (VR) approach, an aspiration point such that, after some time $%
k\geq k_{0}$, there is an advantage to change from any $z^{k}$ to the
aspiration point $x^{\ast }$ (see Soubeyran~\cite{S2010}, Flores et al.~\cite{FLS2012} and Luc and Soubeyran~\cite{LS2013}, for applications). This aspiration
point represents a desirable end for an unbounded sequence going to
infinity. Its existence means that for any trajectory which goes far away
(to infinity), agents can hope to improve their current situation. This is
strongly related to the so called theory of hope (Snyder~\cite{S1994}).

The last hypothesis, given in Definition 4.2, consider well conditionned
bifunctions. It supposes that a non degenerated $F$ is $\rho $-conditioned,
i.e, that there exist positive number $\rho >0$ and $\tau >0$ such that $%
-F(x,P_{S}(x))\geq \tau \mbox{dist}^{\rho }(x,S),x\in \Omega $, where $S=EP(F,\Omega )$%
. This hypothesis is very natural within the variational rationality
context. It means that it is always worthwhile to change from any point $%
x\in $ $\Omega $ to the projection $P_{S}(x)$ of $x$ on the subset $%
S=EP(F,\Omega )$ of equilibrium points. In this case,\ advantages to
change from $x$ to $P_{S}(x)$ are $A(x,P_{S}(x))=-F(x,P_{S}(x))$ and
resistance to change is $R(x,y)=D\left[ \mbox{dist}(x,S)\right]$. This supposes
that any change from $x\in \Omega $ to $P_{S}(x)$ $\in S=EP(F,\Omega )$ is
worthwhile, in a specific sense: $A(x,P_{S}(x))\geq \tau \mbox{dist}^{\rho }(x,S),$
for all $x\in \Omega .$

\subsection{Exact And Approximate Solutions: Reversing The Logic}

In this Section 5, we have shown that the equilibrium problem, with a
Tikhonov regularization term, on an Hadamard manifold, is a particular and
nice instance of a very general variational trap problem, which appears in
Behavioral Sciences in Psychology, Economics, Management Sciences, Game
theory, Decision theory, Philosophy, Artificial Intelligence, Political
Sciences.

The ``variational rationality" approach (see \cite{S2009,S2010}) reverses the
logic of what is an exact or inexact (approximate) solution, taking the
point of view of, either Behavioral Sciences or Mathematics. If we interpret
the famous Tikhonov perturbation term $\langle z-x,y-z\rangle $ of a
pertubed Nash equilibrium problem as a cosinus measure of dissimilarity of
two vectors in an inner product space, and, then, as a dissimilarity cost to
change from one direction of change to an other one, then:
\begin{itemize}
\item [1)] in Behavioral Sciences, the (VR) approach, where costs to be able to
change play a major role and modelize inertia, frictions and learning, the
natural solutions of a perturbed Nash equilibrium problem are variational
traps, reachable, in a worthwhile way, as maximal elements of a relation of
worthwhile changes, not Nash equilibria, which ignore costs to be able to
change (frictions), once you are there, at the Nash equilibrium (where
frictions are absent, but why?). The exact solutions become variational
traps. The justification is that they include costs to be able to change in
their definition. The approximate solutions become Nash equilibria which
ignore costs to be able to change in their strict definition;

\item [2)] in Mathematics, perturbed Nash equilibrium problems have been seen in the
opposite way: variational traps (as solutions of perturbed Nash equilibrium
problems) are seen as approximate solutions of the exact solutions of a Nash
equilibrium problem, with a perturbation term.
\end{itemize}
\section{Conclusions}

In this paper, we provided a sufficient condition to obtain the existence of solutions of EP. We \mbox{presented} a proximal algorithm for EP whose iterative process has been \mbox{considered} in \cite[Theorem 4.10]{CLMM2012}. Our convergence analysis is restricted to Hadamard manifold with null sectional curvature and extends the convergence result presented in \cite{CLMM2012} to the case where the \mbox{bifunction} of EP is not necessarily monotone. We obtained a condition on the bifunction that retrieves the notion of weak sharp minima, and we prove the finite termination of any sequence \mbox{generated} from our iterative process. In particular, the finite termination result presented in~\cite{BC2012} is extended to \mbox{minimization} problems whose constrained set is not necessarily  the whole manifold. We also obtain a finite convergence result for the proximal point method in order to find singularities of single valued \mbox{monotone} vector fields (see~\cite{CFPN2006, LLMM2009} and, hence, for the  variational inequality problem (see~\cite{N2003}). We gave an \mbox{application} to a recent unifying approach of a lot of stability and change theories in Behavioral \mbox{Sciences}, the \textquotedblleft Variational rationality approach of human behavior'', where an equilibrium problem appears to \mbox{modelize} how an agent can reach  a final equilibrium, following a sequence of \mbox{worthwhile} changes and temporary stays from a stationary trap to the next one. As future work, we intend to propose a approach of the proximal point algorithm for equilibrium problems to case that $M$ is an Hadamard manifold with negative sectional curvature, also we intend to establish necessary and sufficient conditions for the hypothesis of our Theorem~\ref{teo4} to occur.


\begin{thebibliography}{999}
\bibitem{BO1994} Blum, E., Oettli, W.: From optimization and variational inequalities to equilibrium problems. Math. Stud. {\bf 63}, 123-145 (1994)

\bibitem{BS1996} Bianchi M., Schaible, S.: Generalized monotone bifunctions and equilibrium problems. J. Optim. Theory  Appl. {\bf 90}, 31-43 (1996)

\bibitem{F1961} Fan, K.: A generalization of Tychonoff's fixed point theorem. Math. Ann. {\bf 142}, 305-310 (1961)

\bibitem{BLS1972} Br\'{e}zis H., Nirenberg, L., Stampacchia, G.: A remark on Ky Fan's minimax principle. Boll. Un. Mat. Ital. {\bf 6}, 293-300 (1972)

\bibitem{IS2003} Iusem, A. N., Sosa, W.: New existence results for equilibrium problems, Nonlinear Anal. {\bf 52}, 621-635 (2003)

\bibitem{IKS2009} Iusem, A. N.,  Kassay, G., Sosa, W.: On certain conditions for the existence of solutions of equilibrium problems. Math. Program., Ser. B, {\bf 116}, 259-273 (2009) 

\bibitem{CLMM2012} Colao, V., L\'opez, G.,  Marino, G., Mart\'in-M\'arquez, V.:  Equilibrium problems in Hadamard manifolds. J. Math. Anal. Appl. {\bf 388}, 61-77 (2012)

\bibitem{M1999} Moudafi, A.:  Proximal  point  algorithm  extended  for  equilibrium  problems.  J.  Nat.  Geom.  {\bf 15}, 91-100 (1999)

\bibitem{K2003} Konnov, I.V.: Application of the proximal method to nonmonotone equilibrium problems. J. Optim. Theory Appl. {\bf 119}, \mbox{317-333} (2003)

\bibitem{IS2010} Iusem, A. N., Sosa, W.: On the proximal point method for equilibrium problems in Hilbert spaces. Optimization {\bf 59}, \mbox{1259-1274} (2010) 

\bibitem{Tang2013} Tang, Guo-Ji, Zhou, Lin-wen, Huang, Nan-jing.: The Proximal point Algorithm for pseudomonotone Variational Inequalities on Hadamard Manifolds. Optim. Lett. {\bf 7}, 779-790 (2013)

\bibitem{FO2002} Ferreira, O.P., Oliveira, P.R.: Proximal point algorithm on Riemannian manifold. Optimization {\bf 51}, 257-270 (2002)

\bibitem{LLMM2009}  Li, C., Lop\'ez, G., Mart\'in-M\'arquez, V.: Monotone vector fields and the proximal point algorithm on Hadamard manifolds. J. London Math. Soc. {\bf 79} (2), 663-683 (2009) 

\bibitem{FLN2005} Ferreira, O.P., Lucambio P\'erez,  L.R., Nemeth, S. Z.: Singularities of monotone vector fields and an extragradient-type algorithm, J. Global Optim. {\bf 31}, 133-151 (2005) 

\bibitem{LLM2010} Li, C., Lop\'ez, G., Mart\'in-M\'arquez, V.: Iterative algorithms for nonexpansive mappings on Hadamard manifolds. Taiwan. J. Math. {\bf 14} (2), 541-559 (2010) 

\bibitem{WLML2010} Wang, J.H.,  L\'{o}pez, G., Mart\'in-M\'arquez, V., Li, C.: Monotone and Accretive Vector Fields on Riemannian
Manifolds. J. Optim. Theory Appl.  {\bf 146 }, 691-708 (2010)

\bibitem{CFL2002} Cruz Neto, J.X., Ferreira, O.P., Lucambio P\'erez,  L.R.: Contributions to the study of monotone vector fields. Acta Math. Hungar. {\bf 94}(4), 307-320 (2002) 

\bibitem{R1997} Rapcs\'{a}k, T.: Smooth Nonlinear Optimization in $\mathbb{R}^n$. Nonconvex optim. Appl. {\bf 19} Kluwer Academic Publishers, Dordrecht, (1997) 

\bibitem{BC2012} Bento, G. C.,  Cruz Neto, J. X.: Finite termination of the proximal point method for convex functions on Hadamard manifolds. Optimization (2012), DOI:10.1080/02331934.2012.730050.

\bibitem{LMWY2011} Li, C.,  Mordukhovich, B. S.,  Wang, J., Yao, J. C.: Weak sharp minima on Riemannian manifolds, SIAM J. Optim. {\bf 21}(4), 1523-1560, (2011)

\bibitem{M2007} Moudafi, A.: On Finite and strong convergence of a proximal method for equilibrium problems. Numer. Funct. Anal. Optim. {\bf 28}, 11-12 (2007)

\bibitem{CFPN2006} Cruz Neto, J.X., Ferreira, O.P., Lucambio P\'erez,  L.R., N\'{e}meth,  S.Z.: Convex and monotone-transformable mathematical programming problems and a proximal-like point method, J. Global Optim. {\bf 35}(1)  53-69 (2006)

\bibitem{N2003} N\'{e}meth, S. Z.: Variational inequalities on Hadamard manifolds. Nonlinear Anal. {\bf 52}, 1491-1498 (2003)

\bibitem{KR2006} K\H{o}szegi, B., Rabin, M.: A model of reference-dependent preferences.
The Q. J. Econ. {\bf 121}(4), 1133-1165 (2006)

\bibitem{S2009} Soubeyran, A.: Variational rationality, a theory of individual stability and change: worthwhile and ambidextry behaviors. Pre-print. GREQAM, Aix Marseillle University (2009)

\bibitem{S2010} Soubeyran A.: Variational rationality and the "unsatisfied man":  routines and the course pursuit between aspirations,  \mbox{capabilities},  beliefs. Preprint GREQAM, Aix Marseillle University (2010)

\bibitem{MP1992} do Carmo, M.P.: Riemannian Geometry. Birkhauser,  Boston (1992)

\bibitem{S1996} Sakai, T.: Riemannian geometry. Transl. Math. Monogr. {\bf 149}, Amer. Math. Soc., Providence, R.I., (1996)

\bibitem{BH1999} Bridson, M., Haefliger, A.: Metric spaces of non-positive curvature. Springer-Verlag, Berlin (1999)

\bibitem{U1994} Udriste, C.: Convex Functions and Optimization Methods on Riemannian Manifolds. Mathematics and Its Applications, {\bf 297}, Kluwer Academic Publishers.  Dordrecht, (1994) 

\bibitem{Bento2012} Bento, G. C., Melo, J. G.: Subgradient method for convex feasibility on Riemannian manifolds. J. Optim. Theory Appl. {\bf 152} (3), 773-785 (2012)

\bibitem{CMSS2014}Cruz Neto, J. X., Melo, I. D., Sousa, P. A., Silva J. P.: About the Convexity of a Special Function on Hadamard Manifolds, preprint (2014), to appear optimization-online, at http://www.optimization-online.org/DB\_FILE/2014/03/4287.pdf


\bibitem{LLMW2011} Li, C., L\'{o}pez, G., Mart\'in-M\'arquez, V., Wang, J.H.: Resolvents of set valued monotone vector Fields in Hadamard manifolds. Set-Valued Var. Anal. {\bf 19} (3), 361-383 (2011) 

\bibitem{AH1996} Alber, S., Heward, W.:  \textquotedblleft GOTCHA!\textquotedblright\
Twenty-five Behavior Traps guaranteed to extend your students' academic and
social skills. Interv. Sch. Clin. {\bf 31}(5), 285-289 (1996)

\bibitem{BW1970} Baer, D., Wolf, M.: \textquotedblleft The entry into natural communities of
reinforcement\textquotedblright. In R. Ulrich, T. Stachnick, \& J. Mabry (Eds.) Control of
human behavior, 319-324. Glenview, IL: Scott Foresman (1970)

\bibitem{P1993} Plous, C.: The psychology of judgment and decision making. McGraw-Hill, New York  (1993)

\bibitem{S2004} Stephen, F.: \textquotedblleft The power of reinforcement". State University of New York Press, 252, Albany (2004) 

\bibitem{B2002} Baumeister, R. F.: Ego depletion and self-control failure: An energy
model of the self's executive function. Self and Identity, {\bf 1} (2), 129-136 (2002)

\bibitem{BH1996} Baumeister, R., Heatherton, T.: Self regulation failure: an overview. Psychol.  Inq. {\bf 7}(1), 1-15 (1996)

\bibitem{HRKR1998} Hammond, J., Ralph, L., Keeney, R., Raiffa, H.: The hidden traps in decision making. Harvard Bus. Rev., september-october (1998)

\bibitem{LM1993} Levinthal, D. A., March, J. G.: The myopia of learning.
Strategic Manage J. \textbf{14}, 95-112 (1993)

\bibitem{A2004} Appadurai, A.: The capacity to aspire: culture
and the terms of recognition. In V. Rao and M. Walton (eds.), Culture and
Public Action 59-84, (Washington, DC: The World Bank) (2004)

\bibitem{HM2006} Heifetz, A., Minelli, E.: Aspiration traps. Mimeo, (2006)

\bibitem{R2006} Ray, D.: Aspirations, Poverty and Economic Change. In A. Banerjee,
R. Benabou and D. Mookherjee (eds.) What we have learnt about Poverty, Oxford University Press (2006)

 
\bibitem{TK1991} Tversky, A., Kahneman, D.:  Loss  aversion in riskless choice: A
reference-dependent model. Q. J. Econ. {\bf 106} (4), 1039-1061 (1991) 

\bibitem{KT1979} Kahneman, D., Tversky, A., Prospect theory: An analysis of
decision under risk. Econometrica Soc. 263-291, (1979)

\bibitem{L1959} Lindblom, C. E.:  The science of \textquotedblleft muddling through\textquotedblright. Public. Admin. Rev. {\bf 19} (2), 79-88 (1959)

\bibitem{S1955} Simon, H.: A Behavioral Model of Rational Choice. Q. J. Econ. {\bf 69}, 99-188 (1955)

\bibitem{B1981} Bratman, M.: Intention and means-end reasoning. Philos. Rev. {\bf 90} (2), 252-265 (1981)

\bibitem{W2000} Wooldridge, M.: Reasoning About Rational Agents. The MIT Press. Manchester (2000)

\bibitem{FLS2012} Flores-Bazan, F., Luc. D, Soubeyran, A.: \textquotedblleft Maximal elements under reference-dependent preferences with applications to behavioral traps and
games". J. Optim. Theory Appl. {\bf 155}(3), 883-901 (2012)

\bibitem{LS2013} Luc, D. T.,  Soubeyran, A.: Variable preference relations: Existence of maximal elements. J. Math. Econ. {\bf 49}(4), 251-262 (2013)

\bibitem{S1994} Snyder, C.:  The psychology of hope: You can get there from here. New York: Free Press (1994)


\end{thebibliography}
\end{document}